\documentclass[11pt]{article}
\usepackage{latexsym,color,amsmath,amsthm,amssymb,amscd,amsfonts,}

\newtheorem{lemma}{Lemma}

\newtheorem{remark}[lemma]{Remark}

\newtheorem*{remark*}{Remark}

\def\upddots{\mathinner{\mkern 1mu\raise 1pt \hbox{.}\mkern 2mu
\mkern 2mu \raise 4pt\hbox{.}\mkern 1mu \raise 7pt\vbox {\kern 7
pt\hbox{.}}} }

\newcommand{\db} {\! \! \!   \! \! \!  \! \! \! \! \! \!}

\newcommand{\F}{F}
\newcommand{\K}{\mathbb K}
\newcommand{\Of}{\mathbb O_{\F}}
\newcommand{\Pf}{\mathbb P_{\F}}

\newcommand{\N}{\mathbb N}
\newcommand{\Z}{\mathbb Z}

\newcommand{\half}{\frac 1 2}

\newcommand{\ab} {{| \!| }}

\newcommand{\Q}{\mathbb Q}
\newcommand{\C}{\mathbb C}

\def\>{\rangle}
\def\<{\langle}

\newtheorem{lem}{Lemma}
\newtheorem{thm}{Theorem}
\newtheorem{cor}{Corollary}

\newtheorem{prop}{Proposition}

\def\dotunion{
\def\dotunionD{\bigcup\kern-9pt\cdot\kern5pt}
\def\dotunionT{\bigcup\kern-7.5pt\cdot\kern3.5pt}
\mathop{\mathchoice{\dotunionD}{\dotunionT}{}{}}} \newtheorem*{theorem*}{Theorem}
\newtheorem*{lemma*}{Lemma}
\newcommand\blfootnote[1]{%
  \begingroup
  \renewcommand\thefootnote{}\footnote{#1}%
  \addtocounter{footnote}{-1}%
  \endgroup
}
\begin{document}
\title{A $p$-adic analog of Hasse-Davenport product relation involving $\epsilon$-factors}
\author{Dani Szpruch}
 \maketitle
\begin{abstract}
In this paper we prove { some generalizations} of the classical Hasse-Davenport product relation for certain arithmetic factors defined on a $p$-adic field $F$, among them one finds the $\epsilon$-factors appearing in Tate's thesis. We then show that these generalizations are equivalent to some representation theoretic identities relating the determinant of ramified local coefficients matrices defined for coverings of $SL_2(F)$ to Plancherel measures and $\gamma$-factors.
\end{abstract}

\section{Introduction}

Let $\K$ be a finite field with $q$ elements. Let $\chi'$ and $\psi'$ be non-trivial characters of $\K^*$ and $\K$ respectively. The  Gauss sum associated with $\chi'$ and $\psi'$ is defined by
$$G(\chi',\psi')=\sum_{x\in \K^*} \chi'(k)\psi'(k).$$
Suppose that $d\in \N$ divides $q-1$. The Hasse-Davenport product relation, \cite{HD} and \cite{GrKo}, states that if $\chi'^d$ is non-trivial then
$$\prod_{\eta \in \widehat{\K^*/{\K^*}^d}} G(\chi'\eta,\psi')=G(\chi'^d,\psi'_d)\prod_{1\neq \eta \in \widehat{\K^*/{\K^*}^d}} G(\eta,\psi').$$
Here $\widehat{\K^*/{\K^*}^d}$ is the subgroup of characters of $\K$ whose kernel contains ${\K^*}^d$ and $\psi'_d$ is the non-trivial character of $F$ defined by $x\mapsto \psi'(dx).$  We note that the quantity $\prod_{1\neq \eta \in \widehat{\K^*/{\K^*}^d}} G(\eta,\psi')$ is quite simple, for instance  if $d$ is odd it equals $q^{\frac {d-1}{2}}$, so this classical identity essentially relates the product on its left hand side to  $G(\chi'^d,\psi'_d)$.

The first goal of our paper is to generalize  the Hasse-Davenport product relation to certain arithmetic factors defined on  $p$-adic fields. The second goal is to give a representation theoretic applications and interpretations to some of these generalizations. We now give more details, starting with the first goal.

Let $F$ be a $p$-adic field and let $\varpi$ be a uniformizer.  Assume that $\K=\Of / \Pf$ is the residue field of $F$ and note that $\K^*=\Of^*/1+\Pf$. Let $\psi$ be a non-trivial character of $F$, let  $\chi$ be a character of $F^*$ and let $s$ be complex number. Let
$$\gamma(s,\chi,\psi)=\epsilon(s,\chi,\psi)\frac{L(1-s,\chi^{-1})}{L(s,\chi)}$$
be  the Tate $\gamma$-factor defined in \cite{T}. Recall that  $\epsilon(s,\chi,\psi)$ is a monomial function in $q^{-s}$ and that if $\chi$ is ramified then $\epsilon(s,\chi,\psi)=\gamma(s,\chi,\psi)$. As we explain in detail in the body of the paper, $\epsilon(s,\chi,\psi)$ and $G(\chi',\psi')$ are analogous as they may be defined as proportion factors arising from similar uniqueness theorems. In fact, $\epsilon$-factors generalize the notion of  Gauss sums: it is well known that
$$\epsilon(1-s,\chi^{-1},\psi)= \bigl(\chi(\varpi)q^{\half-s} \bigr)^{e(\psi)-e(\chi)} \tau(\chi,\psi)$$
where
$$\tau(\chi,\psi)=q^{\frac{-e(\chi)}{2}} \begin{cases} 1  & \chi \, \operatorname{is} \,  \operatorname{unramified} ;\\ \sum_{a\in \Of^*/1+\Pf^{e(\chi)}} \chi(a)\psi(\varpi^{e(\psi)-e(\chi)}a)  & \chi \, \operatorname{is} \,  \operatorname{ramified}. \end{cases}$$
Here $e(\chi)$ and $e(\psi)$ are the conductors of $\chi$ and $\psi$ respectively. If $e(\chi)=1$ then $x\mapsto \psi \bigl(x\varpi^{e(\psi)-e(\chi)}\bigr)$ and $x\mapsto \chi(x)$ are well defined characters on $\K$ and $\K^*$ respectively. Denote these by $\psi'$ and $\chi'$. We have $\tau(\chi,\psi)=q^{-\half}G(\chi',\psi')$.

Suppose now that $d$ is relatively prime to the residual characteristic of $F$. In this case $1+\Pf \subseteq {F^*}^d$. Consequently, There exists a finite group $J$ of characters of $F^*$ such that the restriction to $\Of^*$ defines an isomorphism from $J$ to $\widehat{\K^*/{\K^*}^d}.$ In Section \ref{mainsec} we generalize the Hasse-Davenport product relation for $\tau(\chi,\psi)$ under the assumption that the residual characteristic of $F$ is odd. In our generalization $J$ plays to role of $\widehat{\K^*/{\K^*}^d}$, see  Theorem \ref{tau theorem}.

If $e(\chi)=1$ then Theorem 1 is  the Hasse-Davenport product relation so we focus on the case where $e(\chi)>1.$ In this case, contrary to the $e(\chi)=1$ case, $\prod_{ \eta\in J}{\tau}(\eta\chi ,\psi)$ and $\tau(\chi ,\psi)^d$ are essentially equal and the burden of the proof is shifted to relating $\tau(\chi ,\psi)^d$ to $\tau(\chi^d ,\psi_d)$.

One of the main new ingredients in our proof is a representation of $\tau(\chi,\psi)$ in a certain summation-free form. If $e(\chi)$ is even this was already achieved by Bushnell and Henniart in \cite{BH} building on the relation between $\psi$ and the restriction of $\chi$ to $1+\Pf^{\frac {e(\chi)}{2}}$ described by  Deligne in \cite{del}. We follow these ideas in Lemmas \ref{blemma} and  \ref{tauvalue} for the cases where the conductor of $\chi$ is an odd number greater than 1. { At this moment we could only prove these lemmas under the assumption that the characteristic of $\K$ is odd. Consequently, our main results Theorems \ref{tau theorem}, \ref{mainres} and \ref{ramdet} hold only in this case. However, for the purpose of future generalizations we did formulate some other results in greater generality.}

Our generalization is particulary simple when $d$ is odd, in this case it takes the form:
$$ \prod_{ \eta\in J}{\tau}(\eta\chi ,\psi)=\tau(\chi^d ,\psi_d)$$
and one deduces at once that if $\chi^{d}$ is ramified and $d$ is odd then
\begin{equation} \label{example} \prod_{\eta \in J} \epsilon(1-s,(\chi\eta)^{-1},\psi)=q^{\frac{d-1}{2} \bigl(e(\psi)-e(\chi^d)\bigr)}\epsilon(1-ds,\chi^{-d},\psi_d).\end{equation}

We also prove a similar identity replacing the roll of Tate $\gamma$-factor with the metaplectic $\widetilde{\gamma}$-factor defined in \cite{Sz3}. See Theorem \ref{mainres} in Section \ref{mainsec} for both identities. In Theorem \ref{mainres} we also use the $d=2$ case of the generalized Hasse-Davenport product relation to prove an identity relating $\gamma$ and $\widetilde{\gamma}$-factors.

There is no reason to expect that \eqref{example} holds for unramified characters since the analogy between $\epsilon$-factors and Gauss sums breaks in these cases. For unramified characters the identity is false also when we replace the $\epsilon$-factors by $\gamma$-factors. A Similar failure occurs for the other 2 identities in Theorem \ref{mainres}.

Our second goal is to provide some representation theoretic identities  which do hold in the unramified cases and whose ramified cases are equivalent to those presented in Theorem \ref{mainres}. This is achieved in Section \ref{finalres}, Theorem \ref{ramdet}. We now elaborate on this point for \eqref{example}.

Assume as before that $d$ is { an} odd { number relatively prime to $p$ and that $\mu_d \subseteq F$}. Let $\widetilde{G_o}$ be the familiar $d$-fold cover of $G_o=SL_2(F)$ constructed in  \cite{Kub} (due to some delicacies the notation is slightly different in Section \ref{apsec} where we use $n$ for the degree of the cover). Let $H_o$ be the diagonal subgroup of $G_o$, let $\widetilde{H_o}$ be the inverse image of $\widetilde{H_o}$ inside $\widetilde{G_o}$ and let $N$ be the subgroup of upper triangular matrices in $G$. $\widetilde{G_o}$ splits over $N$ uniquely so we identify $N$ with its embedding in $\widetilde{G_o}$. Since $\widetilde{H_o}$ normalizes this embedding we may extend any representation $\widetilde{H_o}$ to the inverse image of the standard Borel subgroup by defining it to be trivial on $N$.

Denote by $Z(\widetilde{H_o})$ the center of $\widetilde{H_o}$. The isomorphism class of an irreducible admissible genuine representation $(\sigma_o,V_o)$ of $\widetilde{H_o}$ is determined by its central character and the set of genuine characters of $Z(\widetilde{H_o})$ is canonically parameterized by the set of $d^{th}$ powers of characters $F^*$. Denote by $\chi^d$ the character of $F^*$ corresponding to $\sigma_o$ via this parametrization.

Let $I(\sigma_o,s)$ be the normalized parabolic induction on $\widetilde{G_o}$ from $\sigma_o$. Here, $s$ is the usual  complex parameter. Let $A_o({\sigma}_o,s):I(\sigma_o,s) \rightarrow I(\sigma_o^w,{-s})$ be the standard intertwining operator associated with a non-trivial Weyl element $w$. Recall that, $\mu(\sigma_o,s)$, the Plancherel measure associated with $\sigma_o$, is the rational function in $q^{-ds}$ defined by the relation
$$A(\sigma_o^w,{-s})\circ A(\sigma_o,s)=\mu(\sigma_o,s)^{-1}Id.$$

Let $ Wh_\psi\bigl(I(\sigma_o,s)\bigr)$ be the space of $\psi$-Whittaker functionals on $I(\sigma_o,s)$. Unlike the linear case, this space is not one dimensional. By duality, $A(\sigma_o,s)$ give rise to map between the corresponding Whittaker spaces,

$$ {A_o}_{\psi}(\sigma_o,s):Wh_\psi\bigl(I(\sigma_o^w,-s)\bigr) \rightarrow Wh_\psi\bigl(I(\sigma,s)\bigr).$$

As observed in \cite{Sz19} the two Whittaker spaces above are identified by Jacquet-type integrals with the space of functionals on $V_o$ . This means that  ${A_o}_{\psi}(\sigma_o,s)$ may be viewed as an endomorphism rather than a map between two different spaces. As such its determinant $D_o(\sigma_o,s,\psi)$ is an invariant of $\sigma_o$ and $\psi$. In \cite{GSS} we have defined a new invariant by
$$S(\sigma_o,s,\psi)=D_o(\sigma_o,s,\psi) \cdot \mu(\sigma_o,s)^ {\frac{d-1}{2}}.$$
In the linear case, namely in the $d=1$ case, $D_o(\sigma_o,s,\psi)$ is the reciprocal of Shahidi's local coefficient, \cite{Shabook}.
In Theorem 3.14 of \cite{GSS} it was proven that if $\psi$ is normalized, $\chi^d$ is unramified and $d$ is relatively prime to the residual characteristic of $F$ then
$S(\sigma_o,s,\psi)=\gamma(1-ds,\chi^{-d},\psi_{d}).$ Equivalently,
\begin{equation} \label{example2} D_o(\sigma_o,s,\psi)=\mu(\sigma_o,s)^{\frac{1-d}{2}}  \gamma(1-ds,\chi^{-d},\psi_{d}). \end{equation}
This result was generalized in \cite{GSS2} for other Brylinskiy-Deligne covering groups.

In Theorem \ref{ramdet} we shall show that \eqref{example2} holds in the ramified cases as well. Moreover, we shall show that in these cases  \eqref{example2}
is equivalent to \eqref{example}. Hence we claim that the actual generalization of the Hasse-Davenport product relation to $p$-adic fields is Equation \eqref{example2}.

The paper is organized as follows. In Section \ref{pandp} we introduce the  arithmetic objects to be studied and gather some needed information. We emphasize
the relation between $\epsilon$-factors and Gauss sums. In Section \ref{reschi} we prove the needed relations between the restrictions of $\chi$ and $\psi$. We use these relations in Section \ref{comptau} to compute $\tau(\chi,\psi)$ in the cases where $e(\chi)>1$. Our main results, the generalizations of the Hasse-Davenport product relation, are then given in Section \ref{mainsec}. Section \ref{apsec} is devoted to the application of our main results to the representation theory of covering groups. We first describe the required known results on the subject. Then, in Section \ref{finalres} we spell out the formulas for the determinants and connect them to our main results. In the final short section of the paper we give some remarks about the even cases we did not cover.

\section{Preliminaries and preparations} \label{pandp}
\subsection{General notation}
For an { abelian} topological group $G$ we denote by $\widehat{G}$ the group of characters of $G$, that is, the group of continuous homomorphisms from $G$ to $\C^*$. If $H$ is a subgroup of $G$ of finite index we identify $\widehat{G/H}$ with the subgroup of $\widehat{G}$ consisting of elements whose kernel contain $H$. For a complex vector space $V$ we shall denote by $\widehat{V}$ the vector space of linear functionals on $V$.

For a field $L$, $\psi \in \widehat{L}$ and $a\in L^*$ we define  $\psi_a \in \widehat{L}$ by $x\mapsto \psi(ax)$. We note that if $\psi$ is non-trivial then $\psi_a$ is also non-trivial. If $d\in \N$ is relatively prime to the characteristic of $L$ we shall denote by $\mu_d(L)=\mu_d$ the cyclic group of $d$ elements consisting of $d^{th}$ roots of 1 in a fixed algebraic closure of  $L$. Define
$$sign(L)=\begin{cases} 1 &  -1 \in {L^*}^{2}; \\ -1  &   -1 \not \in {L^*}^{2}. \end{cases}$$

\subsection{Gauss sums}
Let $\K$ be a finite field with $q$ elements. Starting from Section \ref{pfield} we shall assume that $\K$ is the residue field of a $p$-adic field $F$.
We shall sometimes  use $'$ when defining objects associated with $\K$ to distinguish them from their $p$-adic counterparts.

Fix $\psi'$, a non-trivial element  in $\widehat{\K}$ and  $\chi'$, a non-trivial element in $\widehat{\K^*}$.  Let
$$G(\chi',\psi')=\sum_{x\in \K^*} \chi'(x)\psi'(x).$$
be the Gauss sum associated with $\psi'$ and   $\chi'$. By changing summation index we obtain
\begin{equation} \label{gausspsi} G(\chi',\psi'_a)=\chi^{-1}(a)G(\chi',\psi'). \end{equation}

We now describe one context in which Gauss sums arise. Let $S(\K)$ be the space of complex functions on $\K$. For $f \in S(\K)$ define its Fourier transform $f_{\psi'} \in S(\K)$ by
$$f_{\psi'}(y)=q^{-\half}\sum_{x\in \K} f(x)\psi'(xy).$$
With this normalization Fourier inversion formula takes the form ${f_{\psi'}}_{\psi'}(x)=f(-x)$.

$\K^*$ acts on  $S(\K)$ by right translations $\rho$. By duality $\K^*$ also acts on $\widehat{S(\K)}$.
Namely, for $\xi \in \widehat{S(\K)}$ and $k \in \K^{*}$ we define $\rho(k)\xi \in \widehat{S(\K)}$ by
\begin{equation} \label{actonhat} \bigl( \rho(k)\xi \bigr)(f)=\xi \bigl(\rho(k)f\bigr). \end{equation}
Let $\widehat{S(\K)}_{\chi'}$ be the space of $\chi$ eigen functionals.
An easy exercise  shows that $\dim \, \widehat{S(\K)}_{\chi'}=1$. This uniqueness fails for the trivial character of $\K^*$.
One then verifies that $$f \mapsto \zeta'(f,\chi')=\sum_{x \in \K^*} f(x)\chi'(x)$$
is a non-zero element in $\widehat{S(\K)}_{\chi'^{-1}}$ and that $f \mapsto \zeta'(f_{\psi'} ,\chi'^{-1})$ is also a non-zero element in $\widehat{S(\K)}_{\chi'^{-1}}$. Thus, there exists a non-zero constant $\epsilon'(\chi',\psi')$ such that
$$\zeta'(f_{\psi'},\chi'^{-1})=\epsilon'(\chi',\psi')\zeta'(f,\chi')$$ for all  $f \in S(\K)$. By plugging $f=\delta_1$ one observes that
\begin{equation} \label{finiteeps} \epsilon'(\chi',\psi')=q^{-\half}G(\chi'^{-1},\psi'). \end{equation}
The well known property,
\begin{equation} \label{gaussinv} G(\chi',\psi')G(\chi'^{-1},\psi')=\chi^{'}(-1)q \end{equation}
follows from Fourier inversion formula. Define
$$G_{\psi'}(\K)=\sum_{x\in \K} \psi'(x^2).$$

If the characteristic of $\K$ is odd then  $\K^*$ has a unique non-trivial quadratic character. We denote it by $\eta_{_\K}$. Note that $sign(\K)=\eta_{_\K}(-1)$.

\begin{lem} If the characteristic of $\K$ is odd then

\begin{equation} \label{quadgauss} G_{\psi'}(\K)=G(\eta_{_\K},\psi'),\end{equation}

\begin{equation} \label{cquadgauss} G_{\psi'_a}(\K)=\eta_{_\K}(a) G_{\psi'}(\K),\end{equation}

\begin{equation} \label{quadinv} G_\psi'(\K)^2=q \cdot sign(\K). \end{equation}

\end{lem}
The first assertion of this Lemma is a standard exercise. The  other two  assertions follow from the first assertion along with   \eqref{gausspsi} and   \eqref{gaussinv}.

Since $\K$ is finite the following 3 assumptions are equivalent: $\widehat{\K^*/{\K^*}^d}$ is a cyclic group of $d$ elements, $\mu_d \subseteq \K$, $d$ divides $q-1$. {In particular, these assumptions imply that $gcd(d,p)=1$.}
\begin{lem} \label{twistprod}Assume that  $\mu_d \subseteq \K$.  Denote by $\eta_{_o}$ a generator of $\widehat{\K^*/{\K^*}^d}$. We have
$$\prod_{1\neq \eta \in \widehat{\K^*/{\K^*}^d}}{G}(\eta,\psi)=  \begin{cases}  q^{\frac {d-1}{2}}  & d \, \operatorname{is} \,  \operatorname{odd} ;\\  q^{\frac {d-2}{2}}\eta_{o}(-1)^{\frac {d(d-2)}{8}}G_\psi(\K)  &  d \, \operatorname{is} \,  \operatorname{even} \end{cases}$$
\end{lem}
\begin{proof} Assume first that $d$ is odd. In this case $\widehat{\K^*/{\K^*}^d}$ has no elements of order 2 and $\eta(-1)=1$ for all $ \eta \in \widehat{\K^*/{\K^*}^d}$. The assertion now follows from \eqref{gaussinv}. Assume now that $d$ is even. In this case { the characteristic of $\K$ is odd and $\eta_{o}^{\frac d 2}=$}$\eta_{_\K}$ is the only element of  $\widehat{\K^*/{\K^*}^d}$ of order 2. Thus, by \eqref{gaussinv} and \eqref{quadgauss}
$$\prod_{1\neq \eta \in \widehat{K^*/{K^*}}^d}{G}(\eta,\psi')=  G_{_\K}(\psi')q^{\frac {d-2}{2}} \prod_{1\leq i\leq \frac{d}{2}-1}\eta_{o}(-1).$$
\end{proof}

\subsection{The classical Hasse-Davenport product relation}

Suppose that $\mu_d { \subseteq} \K$. The classical Hasse-Davenport product relation states that if $\chi'^d$ is non-trivial (equivalently, $\chi' \notin  \widehat{\K^*/{\K^*}^d}$) then
\begin{equation} \label{hd} \prod_{\eta \in \widehat{\K^*/{\K^*}^d}} G(\chi'\eta,\psi')=G(\chi'^d,\psi'_d)\prod_{1\neq \eta \in \widehat{\K^*/{\K^*}^d}} G(\eta,\psi').\end{equation}

\begin{lem} \label{d and sign} Let $\K$ be a finite field of an odd  characteristic. Assume that $\mu_d \subseteq \K$ and  denote by $\eta_o$ a generator of $\widehat{\K^*/{\K^*}^d}$.
\begin{enumerate}
\item  If $d$ id odd then $$\eta_{_\K}(d)=sign(\K)^{\frac {d-1}{2}}.$$
\item If $d$ is even then $$\eta_{_\K}(2d)=\eta_o(-1)^{\frac{d(d-2)}{8}}.$$
{ (here we identify $d$ with an element of $\K$ by takeing its residue mod $p$).}

\end{enumerate}
\end{lem}
\begin{proof} Assume first that $d$ is odd. Using a similar argument to the one we have used in Lemma \ref{twistprod} we obtain
$$\prod_{\eta \in \widehat{\K^*/{\K^*}^d}} G(\eta_{_\K}\eta,\psi')=G(\eta_{_\K},\psi')\prod_{1\neq \eta \in \widehat{\K^*/{\K^*}^d}} G(\eta_{_\K}\eta,\psi')=G(\eta_{_\K},\psi')q^{\frac {d-1}{2}}\eta_{_\K}(-1)^{\frac {d-1}{2}}$$
Also, by \eqref{gausspsi} we have  $$G(\eta_{_\K}^d,\psi'_d)=G(\eta_{_\K},\psi'_d)=\eta_{_\K}(d)G(\eta_{_\K},\psi').$$
The first assertion now follows from Hasse-Davenport product relation \eqref{hd} along with Lemma \ref{twistprod}.

Assume now that $d$ is even and write $d=2^kd'$ where $k>0$ and $d'$ is odd. If $k=1$ we have $\eta_{_\K}(2d)=\eta_{_\K}(d')$ and
$$\eta_o(-1)^{\frac{d(d-2)}{8}}={\bigl(\eta_o(-1)^{\frac d 2}\bigr)}^{\frac{(d-2)}{4}}=\eta_{_\K}(-1)^{\frac{(d'-1)}{2}}.$$
Since $d'$ is odd and $\mu_{d'} \subseteq \K$ the case $k=1$ follows from the first assertion proven in the Lemma. Suppose now that $k=2$. In this case $\frac{d(d-2)}{8}$ is odd and $sign(\K)=1$. Using the first assertion once more we deduce that $\eta_{_\K}(d')=1$. Thus we need to show that $\eta_o(-1)=\eta_{_\K}(2)$. Since $\eta_o(-1)=1$ if and only if $-1\in \mu_d$ it is left to prove that $2\in {\K^*}^2$ if and only if $\mu_8 \subseteq \K.$ Let $i\in {\K^*}$ be a primitive fourth root of 1. Using the identity $(1+i)^2=2i$ we deduce that $2\in {\K^*}^2$ if and only if $i\in {\K^*}^2.$ Finally we prove the case $k\geq 3$. In this case $\frac{d(d-2)}{8}$ is even and as in the $k=2$ case $\eta_{_\K}(d')=1$. Hence it is sufficient to show that  $2\in {\K^*}^2$. This follows since $\mu_8 \subseteq \K.$
\end{proof}

\subsection{p-adic fields} \label{pfield}
Let $\F$ be a finite extension of $\Q_p$. Denote by $\Of$ its ring of integers and fix $\varpi$, a generator of $\Pf$, the maximal ideal of $\Of$. Let  $\K=\Of / \Pf$  be the residue field of $\F$ and let $q$ be its cardinality. Note that $\K^*=\Of^*/ 1+\Pf$. We normalize the absolute value on $F$ so that $\ab \varpi \ab =q^{-1}$.

In this paper, $\psi$ will denote  a non-trivial character of $F$ and $\chi$ will denote a character of $\F^*$. We define  $e(\psi)$ to be the minimal number $n\in \Z$ such that $\psi(\Pf^n)=1$. We note that if $a\in \Of^*$ then $e(\psi)=e(\psi_a)$. In particular, if $d\in \Z$ is relatively prime to the residual characteristic of $F$ then $e(\psi)=e(\psi_d)$. $\psi$ is said to be normalized if $n=0$. We note that $\psi_{\varpi^{e(\psi)}}$ is always normalized.

The map $x\mapsto \psi(x\varpi^{e(\psi)-1})$ is identically 1 on $\Pf$ but not on $\Of$. { Therefore} it defines a non-trivial character  of $\K$. We shall denote it by $\psi'$. By an abuse of language we call $\psi'$ the restriction of $\psi$ to $\K$. We set $G_{\psi}(\K)=G_{\psi'}(\K)$.

For a ramified character $\chi$ of $F^*$ we denote by $e(\chi)$ the minimal number $m\in \N$ such that $\chi(1+\Pf^m)=1$. If $\chi$ in unramified we set $e(\chi)=0$.

If  $e(\chi)\leq 1$  than $x\mapsto \chi(x)$ defines a character $\chi'$ of $\K^*$.  By an abuse of Language we say that $\chi'$ is the restriction of $\chi$. Note that $\chi'$ is trivial if and only if $\chi$ is unramified. Conversely, if $\chi'$ is a character of  $\K^*$ then it may be pulled back to a character of $\Of^*$. We may further extend it to a character of $F^*$ by defiling it arbitrarily on $\varpi$. If $\chi'$ is non-trivial then the conductor of the lifted character is 1.

\begin{lem} \label{Jlemma} Suppose that $d\in \N$ is relatively prime to the residual characteristic of $F$.
\begin{enumerate}
\item $1+\Pf^* \subseteq {F^*}^d.$ Consequently, if $\eta \in \widehat{F^*/{F^*}^d}$ then $e(\eta)\leq 1$ and the restriction of $\eta$ to $\Of^*$ is a pull back of an element of $\widehat{\K^*/{\K^*}^d}$.
\item If $\chi^d$ is ramified then for any $\eta \in  \widehat{F^*/{F^*}^d}$  we have $e(\chi)={e}(\chi^d)=e(\eta\chi)$.
\item $\mu_d(F) \subseteq F$ if and only if $\mu_d(\K) \subseteq \K$ and in this case $[F^*:{F^*}^d]=d^2$. In particular, if the residual characteristic of $F^*$ is odd then $sign(F)=sign(\K)$ and $[F^*:{F^*}^2]=4$.
\item If $\mu_d \subseteq F$ then there exists a subgroup $J$ of $\widehat{F^*/{F^*}^d}$ such that the restriction to $\Of^*$ defines an isomorphism from $J$ to $\widehat{\K^*/{\K^*}^d}.$
\end{enumerate}
\end{lem}
\begin{proof}  These are well known. For the first three items  see for example Lemma 4.15 in \cite{Sz19}. The last item follows at once from the paragraph preceding this lemma.
\end{proof}

We shall refer to the group $J$ whose existence was proven in the last item of the Lemma as a lift of $\widehat{\K^*/{\K^*}^d}$. We note this group is not unique and that one standard choice of a generator $\eta_o$ of $J$ is $x\mapsto (\varpi,x)_d$ where $(\cdot,\cdot)_d$ is the $d^{th}$ power Hilbert symbol. We shall use this choice only in Lemma \ref{goramlcm}.


\subsection{ $\tau$ and its basic properties}
If $\chi$ is ramified  then $x\mapsto \psi \bigl(x\varpi^{e(\psi)-e(\chi)}\bigr)$ and $x\mapsto \chi(x)$ are well defined maps on\\ $\Of^*/1+\Pf^{e(\chi)}$ so we may define

$$\tau(\chi,\psi)=q^{\frac{-e(\chi)}{2}} \begin{cases} 1  & \chi \, \operatorname{is} \,  \operatorname{unramified} ;\\ \sum_{a\in \Of^*/1+\Pf^{e(\chi)}} \chi(a)\psi(\varpi^{e(\psi)-e(\chi)}a)  & \chi \, \operatorname{is} \,  \operatorname{ramified}. \end{cases}$$
Note that for any $\chi$ and $\psi$ we have
\begin{equation} \label{tauunc} \tau(\chi,\psi)=\tau(\chi,\psi_{\varpi^k}).\end{equation}
Thus it is sufficient to study $\tau(\chi,\psi)$ for normalized characters only. Observe also that if  $e(\chi)=1$ then
\begin{equation} \label{gautau} \tau(\chi,\psi)=q^{-\half}G(\chi',\psi') \end{equation}
where $\chi'$ is the restriction of $\chi$ to $\K^*$ and $\psi'$ is the restriction {of} $\psi$ to $\K$.

\begin{cor} \label{hdtau1} Assume that the residual characteristic of $F$ is odd. Assume also that $d$ is relatively prime to the residual characteristic of  $F$ and that $\mu_d \subseteq F$. Let $J \in\widehat{ F^*/{F^*}^d}$ be a lift of $\widehat{ {\K}^*/{\K^*}^d}$. Assume that $e(\chi)\leq 1$. Then
$$\prod_{\eta \in J}\tau(\chi \eta,\psi)=  \tau(\chi^d,\psi_d) \begin{cases} 1 & d \, \operatorname{is} \,  \operatorname{odd} ;\\  \eta_{_\K}(2d)q^{-\half}G_\psi(\K)  &  d \, \operatorname{is} \,  \operatorname{even} \end{cases}$$
\end{cor}

\begin{proof}
If $\chi^d$ is unramified then the restriction of $\chi$ to $\Of^*$ lies in $J$. Therefore
 $$\prod_{\eta \in J}\tau(\chi \eta,\psi)=\prod_{\eta \in J}\tau(\eta,\psi)=\prod_{1 \neq \eta \in J}\tau(\eta,\psi).$$
Since $\tau(\chi^d,\psi_d)=1$ the lemma now follows from \eqref{gautau} along with Lemma \ref{twistprod}  and the second item in Lemma \ref{d and sign}.
Assume that  $\chi^d$ is ramified. In this case, by using   \eqref{gautau} and  the second item in Lemma \ref{d and sign} once more the Lemma follows the Hasse-Davenport product relation \eqref{hd}.
\end{proof}

\subsection{Tate $\gamma$ and $\epsilon$-factors.} \label{tp}
In this Section we recall the definition of Tate $\gamma$ and $\epsilon$-factors along with some of their basic properties. The standard references are \cite{T} and \cite{Tate79}.

Let $S(F)$ be the space of Schwartz functions on $\F$.  For $\phi \in S(F)$ let $\phi_\psi \in S(F)$  be its $\psi$-Fourier transform, i.e.,
$$\phi_\psi(x)=\int_F \phi(y) \psi(xy) \, d_\psi y.$$
Here $d_\psi y$ is the $\psi$-self dual Haar measure on $F$. In other words, $d_\psi y$ is the unique  Haar measure on $F$ such that $\left(\phi_\psi \right)_{{\psi}}(x)=\phi(-x)$ for all $\phi \in S(F), x\in \F$. If $\psi$ is normalized then $\int_{\Pf^m}\, d_\psi x=q^{-m}.$ We set $d^*_\psi x=\frac {d_\psi x}{\ab x \ab}.$ It is a Haar measure on $F^*$.

$F^*$ acts on $S(F)$ and on $\widehat{S(F)}$ { in} the same way as in \eqref{actonhat}. The space $\widehat{S(F)}_{\chi}$  of $\chi$ equivariant functionals on $S(F)$, is one dimensional, see \cite{Kudla} for example. Note that contrary to the finite field case, this uniqueness holds also for the trivial character. For $s\in \C$ define $\chi_s$  to be the unramified twist of $\chi$   given by
$$x \mapsto \ab x \ab^s\chi(x).$$
If $Re(s)>>0$ then for all $\phi \in S(F)$,
$$\int_{F^*} \phi(x) \chi_s(x) \, d^*x$$
converges to a rational function in $q^{-s}$.  Its meromorphic continuation is denoted by $\zeta(s,\chi,\phi).$
This function has ${L(s,\chi)}^{-1}$ as  "common denominator"  where the local $L-$function is defined by $$L(s,\chi)=\begin{cases}  \frac {1} {1-q^{-s}\chi(\varpi)}\ & \chi \, \operatorname{is} \,  \operatorname{unramified} ;\\  1  &  \operatorname{otherwise.} \end{cases}$$
In other words, $$\phi \mapsto {L(s,\chi)}^{-1} \zeta(s,\chi,\phi)$$
is a non-zero element in $\widehat{S(F)}_{\chi_s}$.
Since $$\phi \mapsto {L(1-s,\chi^{-1})}^{-1} \zeta(1-s,\chi^{-1},\phi_\psi)$$
is another non-zero element in  $\widehat{S(F)}_{\chi_s}$  there exists a non-zero
monomial function in $q^{-s}$  denoted by $\epsilon(s,\chi,\psi)$
such that for all $\phi \in S(F)$,
\begin{equation} \label{funeqe} {L(1-s,\chi^{-1})}^{-1}\zeta(1-s,\chi^{-1}, \phi_\psi)=\epsilon(s,\chi,\psi){L(s,\chi)}^{-1}\zeta(s,\chi,\phi). \end{equation}
Tate $\gamma$-factor is defined by
\begin{equation} \label{funeqg} \gamma(s,\chi,\psi)=\epsilon(s,\chi,\psi)\frac{L(1-s,\chi^{-1})}{L(s,\chi)}. \end{equation}
 So it  satisfies the  functional equation

$$ \zeta(1-s,\chi^{-1},\phi_\psi)=\gamma(s,\chi,\psi)\zeta(s,\chi,\phi).$$
The following are easy and well known properties of $\epsilon$-factors: $\epsilon(s,\eta,\psi)$=1 if $\eta$ is unramified and $\psi$ is normalized. Also,
\begin{eqnarray} \label{Tate gamma} {\epsilon}(1-s,\chi^{-1},\psi) &=& \chi(-1)\epsilon(s,\chi,\psi)^{-1}, \\
\label{changepsi}\epsilon(s,\chi,\psi_a) &=& \chi(a)\ab a \ab^{s-\half}\epsilon(s,\chi,\psi), \\
\label{epsilon old twist} \epsilon(s+t,\chi,\psi) &=& q^{\bigl(e(\psi)-e(\chi)\bigr)t}\epsilon(s,\chi,\psi),\\
\label{epsilon twist} \epsilon(s,\chi\eta,\psi) &=&  \eta(\varpi)^{e(\chi)-e(\psi)}\epsilon(s,\chi,\psi),\end{eqnarray}
where $\eta$ is an unramified character.

If $\chi$ is ramified then \eqref{funeqe} and \eqref{funeqg} coincide and  $\gamma(s,\chi,\psi)=\epsilon(s,\chi,\psi)$. Hence, in these cases, the definition of  $\epsilon(s,\chi,\psi)$ and $\epsilon'(\chi',\psi')$ are {completely} analogues. Furthermore, as we shall now show, $\epsilon'${-factors} are special values of $\epsilon$-factors.

\begin{prop} \label{eptau}
$$\epsilon(1-s,\chi^{-1},\psi)= \bigl(\chi(\varpi)q^{\half-s} \bigr)^{e(\psi)-e(\chi)} \tau(\chi,\psi).$$
\end{prop}
\begin{proof} This is well known. The proof is included here for completeness. Suppose first that $\chi$ is unramified. Since $\psi_o=\psi_{\varpi^{e(\psi)}}$ is normalized it follows that $\epsilon(s,\chi,\psi_o)=1$. Hence, the assertion follows from  \eqref{changepsi}.

Suppose now that $\chi$ is ramified. Due to \eqref{changepsi} and to \eqref{tauunc} it is sufficient to prove the assertion for normalized $\psi$. In this case, from Page 14 of \cite{Tate79} it follows that
$$ \epsilon(1-s,\chi^{-1},\psi)=\bigl(\chi(\varpi)q^{1-s} \bigr)^{-e(\chi)} \int_{\Of^*} \chi(u)\psi(\varpi^{-e(\chi)}u) \, d_\psi u. $$
We write
$$\int_{\Of^*} \chi(u)\psi(\varpi^{-e(\chi)}u)= \sum_{a\in \Of^*/1+\Pf^{e(\chi)}} \int_{1+\Pf^{e(\chi)}} \chi(au)\psi(\varpi^{-e(\chi)}au) \, d_\psi u.$$
For $a\in \Of^*$, the maps $u\mapsto \chi(au)$ and $u\mapsto \psi(\varpi^{-e(\chi)}au)$ defined on $1+\Pf^{e(\chi)}$ are constant maps. Since $\int_{1+\Pf^{e(\chi)}}\, d_\psi u=q^{-e(\chi)}$ the assertion follows.
\end{proof}

Proposition \ref{eptau} and Equation \eqref{gautau} show that if $e(\chi)=1$ then $\epsilon(s, \chi,\psi)$ is essentially a Gauss sum. Moreover, let $\chi'$ be a non-trivial character of $\K^*$. Lift $\chi'$ to a character $\chi$ of $F^*$ by setting $\chi(\varpi)=1$. Let $\psi'$ be the restriction of $\psi$ to $\K$. By Proposition \ref{eptau} and by \eqref{gautau} $\epsilon(\half, \chi,\psi)=\epsilon'(\chi',\psi')$ (the appearance of the $\half$ here arises from the fact that the action of $F^*$ on $S(F)$ is not unitary).  We note that the analogy between  $\epsilon$ and $\epsilon'$ fails for unramified characters of $F^*$ as the uniqueness giving rise to $\epsilon'$ breaks for the trivial character of $\K$.

\subsection{Weil index and the metaplectic $\widetilde{\gamma}$-factor}

For $a \in F^*$ let $\eta_a$ be the quadratic character of $F^*$ whose kernel is $N\bigl(F(\sqrt{a})\bigl)$. The map $a\mapsto \eta_a$ is an isomorphism from $F^*/{F^*}^2$ to its dual and  $\eta_a(b)=\eta_b(a)$, see \cite{FV}  Chapter IV, Section 5 for example. Let
$$\gamma(\psi)=\ab 2 \ab^{\half} \lim_{r \rightarrow \infty} \int_{\Pf^{-r}} \psi(x^2) d_\psi x. $$
be the Weil index defined in \cite{Weil65}. It is an eighth root of 1.
For $a\in F^*$ define the normalized Weil index
$$\gamma_\psi(a)=\frac{\gamma(\psi_a)}{\gamma(\psi)}.$$
Observe that since $\gamma(\psi)^{-1}=\overline{\gamma(\psi)}=\gamma(\overline{\psi})=\gamma(\psi_{-1})$ it follows that $\gamma(\psi)^2=\gamma_\psi(-1)$.

It  was proven in Section 14 of \cite{Weil65} that $\gamma_\psi(ab)=\gamma_\psi(a)\gamma_\psi(b)\eta_a(b) $ but we shall not use directly this important property. By \cite{Kahn} and \cite{Kahn87}

\begin{equation} \label{kahn} \gamma_\psi(a)=\epsilon(\half,\eta_a,\psi_{-1})\end{equation}
(see also \cite{Sz9} for a simpler proof).

If the residual characteristic of $F$ is odd then $\eta_a$ is unramified is and only if the valuation of $a$ is even. In particular  $\eta_a(\varpi)=\eta_{_\K}(a)$  for all  $a\in \Of^*$.  It now follows from Proposition \ref{eptau} and \eqref{kahn} that in the odd residual characteristic case  for  $a\in \Of^*$ we have $\gamma_\psi(a)=\eta_{_\K}(a)^{e(\psi)}$. By the above, in this case we also have $\gamma(\psi)^4=\gamma_\psi(-1)^2=1$.

\begin{lem} \label{weild} Let $F$ be a p-adic field of odd residual characteristic and let $d$ be an  odd integer, relatively prime to the residual characteristic of  $F$. If $\mu_d \subseteq F$ then $\gamma(\psi)^d=\gamma(\psi_d).$
\end{lem}
\begin{proof}
Since $\gamma(\psi)^2=\gamma_\psi(-1)$ we have $$\gamma(\psi)^d={\gamma(\psi)^{2}}^{\frac {d-1}{2}}\gamma(\psi)=\eta_{_\K}(-1)^{e(\psi)\frac {d-1}{2}}\gamma(\psi).$$
Thus,
$$\frac{\gamma(\psi_{d})}{\gamma(\psi^d)}=\eta_{_\K}(-1)^{e(\psi) \frac {d-1}{2}}\gamma_\psi(d)=\eta_{_\K}(-1)^{e(\psi)\frac {d-1}{2}}\eta_{_\K}(d)^{e(\psi)}.$$
With the first item in Lemma \ref{d and sign} we now finish.
\end{proof}
By modifying the functional equation giving rise to Tate $\gamma$-factor we have defined in \cite{Sz3} another factor, $\widetilde{\gamma}(s,\chi,\psi)$, which we call the metaplectic $\widetilde{\gamma}$-factor.  Similar to Tate $\gamma$-factor it is the meromorphic continuation of a certain principal value integral. This integral was computed in an unpublished note of J.Sweet, \cite{Sweet}. The computation was reproduced in the appendix of \cite{GoSz}:

$$ \widetilde{\gamma}(1-s,\chi^{-1},\psi)=\gamma(\psi) \chi(-1) \frac{ \gamma(s+\half,\chi,\psi)}{ \gamma(2s,\chi^{2},{\psi_{_2}})}$$
(we refer the reader also to \cite{Sz22} for a shorter proof of this identity). We now define
\begin{equation} \label{eptildef} \widetilde{\epsilon}(1-s,\chi^{-1},\psi)=\gamma(\psi) \chi(-1) \frac{ \epsilon(s+\half,\chi,\psi)}{ \epsilon(2s,\chi^{2},{\psi_{_2}})} \end{equation}
so that $$ \widetilde{\gamma}(1-s,\chi^{-1},\psi)= \widetilde{\epsilon}(1-s,\chi^{-1},\psi)\frac{L(\half-s,\chi^{-1})L(2s,\chi^{2})}{L(\half+s,\chi)L(1-2s,\chi^{-2})}.$$
Observe that if $\chi^2$ is ramified then $\widetilde{\epsilon}(1-s,\chi^{-1},\psi)=\widetilde{\gamma}(1-s,\chi^{-1},\psi).$
\begin{lem} \label{verifygao} (\cite{GSS2}, Corollary 4.5)
If $F$ has an odd residual characteristic, $\psi$ is normalized and $\chi^2$ is unramified then
$$\prod_{\beta \in\widehat{ F^* /{F^*}^2}} \widetilde{\gamma}(1-s,(\chi\beta)^{-1},\psi)=sign(F)\gamma(1-2s,\chi^{-2},\psi_{2})^2 \frac{L (2s,\chi^2 )L (-2s,\chi^{-2}\bigr)}{L (1-2s,\chi^{-2} )L (1+2s,\chi^2)}. $$
\end{lem}
{(according to  \cite{GSS2}, Corollary 4.5 we should have  $\psi$ rather than $\psi_2$ in the right hand side. The change is justified by \eqref{changepsi}).}

\section{Relations between $\psi$ and some restrictions of $\chi$} \label{reschi}
In this {section} and in Section \ref{comptau} we assume that $\psi$ is normalized and that $e(\chi)=m\geq 2$. In these two sections we shall denote by  $k$ the largest integer satisfying $m\geq 2k$, namely,

$$k= \begin{cases}  \frac {m} {2}\ & m \, \operatorname{is} \,  \operatorname{even} ;\\   \frac {m-1}{2}  &  m \, \operatorname{is} \,  \operatorname{odd} . \end{cases}$$
Define
$$h_{\psi}: 1+\Pf^{m-k} /1+\Pf^m \rightarrow \C^*$$ by
  $$h_{\psi}(x)=\psi_{\varpi^{-m}}(x-1)=\frac 1 {\psi(\varpi^{-m})}\psi(\varpi^{-m}x).$$
\begin{lem}  \label{h lem}  The following hold.
\begin{enumerate}
\item $h_{\psi}$ is a well defined character of $1+\Pf^{m-k} /1+\Pf^m$.
\item The restriction of $h_{\psi}$ to $1+\Pf^{m-1} /1+\Pf^m$ is non-trivial.
\item For $a,b \in \Of^*$, $h_{\psi_a}=h_{\psi_b}$ if and only if $a \equiv b \, \, {mod} \, \, 1+ \Pf^k$.
\end{enumerate}
\end{lem}
\begin{proof}
Given $x,y \in1+\Pf^{m-k}$ write $x=1+u\varpi^{m-k}, \, y=1+t\varpi^{m-k}$ where $u,t \in \Of$. Observe that $$h_{\psi}(x)=\psi(u\varpi^{-k}), \, h_{\psi}(y)=\psi(t\varpi^{-k}).$$
\begin{enumerate}
\item Since $xy-1=(u+t)\varpi^{m-k}+ut\varpi^{2m-2k}$ and since $m-2k\geq 0$ we have
$$h_{\psi}(xy)=\psi\bigl((u+t)\varpi^{-k}\bigr)=\psi(u\varpi^{-k})\psi(t\varpi^{-k}).$$
If $y \in 1+\Pf^m$ then $t\in \Pf^k$ so $h_{\psi}(xy)=h_{\psi}(x)$. This shows that $h_{\psi}$ is well defined. It is also clear that $h_{\psi}(xy)=h_{\psi}(x)h_{\psi}(y)$ for all $x,y \in1+\Pf^{m-k}$.

\item If $x\in 1+\Pf^{m-1}$ then $u=u'\varpi^{k-1}$ where $u' \in \Of$. Therefore,   $h_{\psi}(x)= \psi \bigl(\varpi^{-1}u' \bigr)$. The assertion follows.
\item Observe that  $$h_{\psi_a}(x)h_{\psi_b}^{-1}(x)=\psi\bigl(\varpi^{-k}u(a-b) \bigr)=\psi_b\bigl(\varpi^{-k}u(ab^{-1}-1) \bigr).$$
The assertion follows.
\end{enumerate}
\end{proof}
\begin{cor} \label{c cor} (\cite{del}, Lemma 4.16) There exists a unique element $c_{\chi,\psi} \in  \Of^*/1+\Pf^k$ such that
\begin{equation} \label{c def} h_{\psi_{c_{\chi,\psi}}}(x)=\chi^{-1}(x) \end{equation} for all $x\in 1+\Pf^{m-k} /1+\Pf^{m}$.
\end{cor}
\begin{proof} Denote by $A$ the set of characters of $1+\Pf^{m-k} /1+\Pf^m$ whose restriction to \\$1+\Pf^{m-1} /1+\Pf^m$ is non-trivial. Since $x\mapsto \chi^{-1}(x)$ is an element of $A$ we need to prove that $a\mapsto h_{\psi_a}$ is a bijection from $ \Of^*/1+\Pf^k$ to $A$. By the second and third items in Lemma \ref{h lem} we know that this map is well defined and that it is one to one. It is left to show that $A$ has the same cardinality as $\Of^*/1+\Pf^k$. Recall that if $H$ is a subgroup of a finite  { abelian} group $G$ then any character on $H$ has $[G:H]$ extensions to a character of $G$. Thus, the cardinality of $A$ is $q^k-q^{k-1}$ as required.
\end{proof}
\begin{lem} \label{cprop} $c_{\chi,\psi}$ has the following properties
\begin{enumerate}
\item $c_{\chi,\psi}=c_{\chi,\psi_y}$ for all $y\in 1+\Pf^k$.
\item For any $l\in \Z$
\begin{equation} \label{c k prop} c_{\chi,\psi}=c_{\chi^l,\psi_l}.\end{equation}
\item For $a\in \Of^*$ we have
\begin{equation} \label{c shift prop} c_{\chi,\psi_a}=a^{-1}c_{\chi,\psi}. \end{equation}
\item { I}f $\eta$ is a character of $F^*$ such that $e(\eta)\leq m-k$ then  $c_{\eta\chi,\psi}=c_{\chi,\psi}$
\end{enumerate}
\end{lem}
\begin{proof}
The first asterion follows from the fact that  $h_{\psi}=h_{\psi_y}$
for all $y\in 1+\Pf^k$, see the last item in Lemma \ref{h lem}. The second assertion is proven by taking the $l^{th}$ power of both sides of \eqref{c def}. For the third assertion just note that
$$ h_{\psi_{ac_{\chi,\psi_a}}}(x)=\chi^{-1}(x)$$ for all $x\in 1+\Pf^{m-k} /1+\Pf^{m}$. Equivalently, $ac_{\chi,\psi_a}=c_{\chi,\psi}.$ The last statement follows from the fact that the restrictions of $\chi$ and $\eta\chi$ to $1+\Pf^{m-k}$ are equal.
\end{proof}
\begin{lem} \label{blemma} Assume that both $m$ and the residual characteristic of $F$  are odd. Then, there exists a unique element $b_{\chi,\psi}$ in $\Of / \Pf$ such that for all $x\in 1+\Pf^{\frac{m-1}{2}}$

$$\chi(x)=\psi_{2^{-1}c_{\chi,\psi}\varpi^{-m}}\bigl((x-1)^2-2(x-1)\bigr)\psi_{c_{\chi,\psi}b_{\chi,\psi}\varpi^{-\frac{m+1}{2}}}(x-1).$$
\end{lem}
\begin{proof}
Similar to the proof of Lemma \ref{h lem} one shows that the map $x\mapsto \psi_{\varpi^{-\frac{m+1}{2}}}(x-1)$ defines a character on $1+\Pf^{\frac{m-1}{2}}$ whose kernel contains $1+\Pf^{\frac{m+1}{2}}$ and that for $a,b \in \Of$ the maps
 $x\mapsto \psi_{a\varpi^{-\frac{m+1}{2}}}(x-1)$ and  $x\mapsto \psi_{b\varpi^{-\frac{m+1}{2}}}(x-1)$ defined  on $1+\Pf^{\frac{m-1}{2}}$ are equal if and only if  $a \equiv b \, \, {mod} \, \, \Pf$.

Since any character of $1+\Pf^{\frac{m+1}{2}}$ has exactly $q$ extensions to a character of $1+\Pf^{\frac{m-1}{2}}$ and since $\chi \mid_{ 1+\Pf^{\frac{m+1}{2}}}=h_{\psi_{-c_{\chi,\psi}}}$ it is left to show that the map
$$x\mapsto \alpha(x)=\psi_{2^{-1}c_{\chi,\psi}\varpi^{-m}}\bigl((x-1)^2-2(x-1)\bigr)$$ defined on $1+\Pf^{\frac{m-1}{2}}$  is a character and that its restriction to $1+\Pf^{\frac{m+1}{2}}$ equals $h_{\psi_{-c_{\chi,\psi}}}$.

We start with the second assertion. Given $x\in 1+\Pf^{\frac{m-1}{2}}$ write $x=1+t\varpi^{\frac{m-1}{2}}$ where $t\in \Of$. With this notation
$$\alpha(x)=\psi_{2^{-1}c_{\chi,\psi}\varpi^{-m}}\bigl( t^2\varpi^{m-1}-2t\varpi^{\frac{m-1}{2}}\bigr).$$
Suppose now that $x\in 1+\Pf^{\frac{m+1}{2}}$. Then $t=t'\varpi$ where $t'\in \Of$. Since $2^{-1}\in \Of^*$, as the residual characteristic of $F$ is assumed to be odd, it follows that $\psi_{2^{-1}}$ is normalized so we have
$$\alpha(x)=  \psi_{2^{-1}c_{\chi,\psi}\varpi^{-m}}\bigl( t'^2\varpi^{m+1}-2t'\varpi^{\frac{m+1}{2}}\bigr)= \psi_{-c_{\chi,\psi}\varpi^{-m}}\bigl(t'\varpi^{\frac{m+1}{2}}\bigr)= \psi_{-c_{\chi,\psi}\varpi^{-m}}(x-1).$$

We now show that $\alpha$ is a character. Write  $x=1+t\varpi^{\frac{m-1}{2}}, \, y=1+z\varpi^{\frac{m-1}{2}}$ where $t,z\in \Of$.
By the above $$\alpha(x)\alpha(y)=\psi_{2^{-1}c_{\chi,\psi}\varpi^{-m}}\bigl( (t^2+z^2)\varpi^{m-1}-2(t+z)\varpi^{\frac{m-1}{2}}\bigr).$$
On the other hand
$$xy=1+\varpi^{\frac{m-1}{2}}(z+t+zt\varpi^{\frac{m-1}{2}}),$$
Hence,
$$\alpha(xy)=\psi_{2^{-1}c_{\chi,\psi}\varpi^{-m}}\bigl( (z+t+zt\varpi^{\frac{m-1}{2}})^2\varpi^{m-1}-2(z+t+zt\varpi^{\frac{m-1}{2}})\varpi^{\frac{m-1}{2}}\bigr).$$
Since
\begin{eqnarray}  \nonumber (z+t+zt\varpi^{\frac{m-1}{2}})^2 &\varpi^{m-1} &-2(z+t+zt\varpi^{\frac{m-1}{2}})\varpi^{\frac{m-1}{2}}= \\  \nonumber
z^2\varpi^{m-1}+t^2 \! \! \! \! & \varpi^{m-1} &+\bigl(z^2t^2\varpi^{2(m-1)}+2z^2t\varpi^{\frac{3(m-1)}{2}}+2zt^2\varpi^{\frac{3(m-1)}{2}}\bigr)-
2(z+t)\varpi^{\frac{m-1}{2}} \end{eqnarray}
it follows that
$$\alpha(xy)=\psi_{2^{-1}c_{\chi,\psi}\varpi^{-m}}\bigl((z^2+t^2)\varpi^{m-1}-2(z+t)\varpi^{\frac{m-1}{2}}\bigr).$$
\end{proof}
\begin{remark} The formula proven in Lemma \ref{blemma} might seem unexplained. However it is closely related to the following observation. On the set $\K \times \K$ one defines an abelian group structure by defining $(x,y)(a,b)=(a+x,y+b+ax).$ Denote this group by $A$ (it is isomorphic to a certain { maximal} abelian subgroup of the Heisenberg group). Note that for $l\geq 1$ the group $1+\Pf^l/1+\Pf^{l+2}$ is isomorphic $A$. Lemma \ref{blemma} essentially gives a formula for extending a non trivial character  of the subgroup $(0,\K)\simeq \K$ of $A$ to $A$. The key fact is that if the residual characteristic of $F$ is odd then $(x,y)\mapsto(x,y-\frac{x^2}{2})$ is an isomorphism from $A$ to the group  $\K \times \K$. If the residual characteristic of $F$ is even then $A$ and  $\K \times \K$ are not isomorphic since $(1,1)\in A$ has order 4 while all the non-trivial elements in $\K \times \K$ has order 2. At this point we do not know how to solve this extension problem in the even cases. Consequently, we had to exclude  these cases in Lemma \ref{blemma}.
\end{remark}
\begin{lem} \label{bprop} $b_{\chi,\psi}$ has the following properties
\begin{enumerate}

\item  $b_{\chi,\psi}=b_{\chi^l,\psi_l}.$

\item If $\eta$ is a character of $F^*$ such that $e(\eta) \leq \frac{m-1}{2}$ then  $b_{\eta\chi,\psi}=b_{\chi,\psi}$.
\end{enumerate}
\end{lem}
\begin{proof}
This follows by using similar reasoning as in the proof of Lemma \ref{cprop}.
\end{proof}

\section{Computations of $\tau(\chi,\psi)$ in the case where $e(\chi)\geq 2$} \label{comptau}


\begin{lem} (\cite{BH}, Section 23.6) \label{tdelta}
\begin{equation} \label{taufor} \tau(\chi ,\psi)=q^{k-\frac m 2}\chi(c_{\chi,\psi})\sum_{y\in 1+\Pf^{k}/1+\Pf^{m-k}}\chi(y) \psi(c_{\chi,\psi}y\varpi^{-m}).\end{equation}

\end{lem}
\begin{proof} We write
$$\tau(\chi ,\psi)=q^{\frac{-m}{2}}\sum_{a\in \Of^*/1+\Pf^k} \chi (a) \delta(a,\chi,\psi) $$
where for $a\in \Of^*/1+\Pf^k$ $$ \delta(a,\chi,\psi)=\sum_{x\in 1+\Pf^{k} /1+\Pf^m} \chi(x)\psi_a(\varpi^{-m}x).$$
The proof of this lemma is completed once we show that
$$\delta(a,\chi,\psi)=\begin{cases} q^k \sum_{y\in 1+\Pf^{k}/1+\Pf^{m-k}}\chi(y) \psi(c_{\chi,\psi}y\varpi^{-m}) &  a= c_{\chi,\psi}; \\ 0  &   a\neq c_{\chi,\psi}. \end{cases}$$
For this purpose we further write
\begin{eqnarray}  \nonumber \delta(a,\chi,\psi) =\sum_{y\in 1+\Pf^{k}/1+\Pf^{m-k}}\db &\chi(y)& \db \quad \sum_{x\in 1+\Pf^{{m-k}} / 1+\Pf^m} \db \chi(x)\psi_{ay}(\varpi^{-m}x)=  \\ \nonumber  \sum_{y\in 1+\Pf^{k}/1+\Pf^{m-k}} \db &\chi(y)& \psi(ay\varpi^{-m}) \db \sum_{x\in 1+\Pf^{{m-k}} / 1+\Pf^m} \db\chi(x)h_{\psi_{ay}}(x).\end{eqnarray}

By the third item of Lemma \ref{h lem} we may replace the $h_{\psi_{ay}}(x)$ in the inner summation by $h_{\psi_{a}}(x)$. By Corollary \ref{c cor} and by the orthogonality relation of characters of finite { abelian} groups
\begin{equation} \label{deltacomp} \sum_{x\in 1+\Pf^{m-k} /1+\Pf^m} \chi(x) h_{\psi_{ay}}(x)= \begin{cases} q^k  &  a=c_{\chi,\psi}; \\ 0  &   \operatorname{otherwise}. \end{cases} \end{equation}
\end{proof}
\begin{cor} \label{taulowtwist}
If $\eta$ is a character of $F^*$ such that $e(\eta)\leq k$ then  $$\tau(\chi\eta ,\psi)=\eta(c_{\chi,\psi})\tau(\chi,\psi).$$
\end{cor}
\begin{proof} This follows at once from Lemma \ref{tdelta} combined with the last item in Proportion \ref{cprop}.
\end{proof}
We Note that using similar arguments, a similar statement for $\epsilon$-factors is proven in \cite{del}, Lemma 4.16.
\begin{lem} \label{tauvalue}
If $m$ is even then
$$\tau(\chi ,\psi)=\chi (c_{\chi,\psi})\psi(c_{\chi,\psi}\varpi^{-m}).$$
If both $m$ and the residual characteristic of $F$ are odd then
$$ \tau(\chi ,\psi)= \chi (c_{\chi,\psi}) \psi(c_{\chi,\psi}\varpi^{-m})\psi_{2^{-1}c_{\chi,\psi}\varpi^{-1}}\bigl(-b^2_{\chi,\psi}\bigr)\eta_{_\K}(2c_{\chi,\psi})q^{\frac{-1}{2}}G_\psi(\K).$$

\end{lem}
\begin{proof}
We shall compute the sum in the right hand side of \eqref{taufor} . If $m$ is even the summation is taken over the trivial group and there is nothing to prove. We assume that $m$ is odd. In this case the summation is taken over the group $1+\Pf^{\frac{m-1}{2}}/1+\Pf^{\frac {m+1}{2}} \simeq \K$. To facilitate the reading we shall denote this group in this proof by $c_m$. We have

$$\tau(\chi ,\psi)= q^{\frac{-1}{2}}\chi(c_{\chi,\psi}) \sum_{y\in c_m}   \chi(y) \psi(c_{\chi,\psi}y\varpi^{-m}).$$
We need to  show that
$$ \sum_{y\in c_m}   \chi(y) \psi(c_{\chi,\psi}y\varpi^{-m})=\psi(c_{\chi,\psi}\varpi^{-m})\psi_{2^{-1}c_{\chi,\psi}\varpi^{-1}}\bigl(-b^2_{\chi,\psi}\bigr)\eta_{_\K}(2c_{\chi,\psi})G_\psi(\K).$$
We write
$$\sum_{y\in c_m}   \chi(y) \psi(c_{\chi,\psi}y\varpi^{-m})= \psi(c_{\chi,\psi}\varpi^{-m})   \sum_{y\in c_m}   \chi(y) \psi_{c_{\chi,\psi}\varpi^{-m}}(y-1).$$
By Lemma \ref{blemma},

\begin{equation}\nonumber
\begin{split}
 & \sum _{y\in c_m} \chi(y)  \psi(c_{\chi,\psi}y\varpi^{-m})= \\
\psi(c_{\chi,\psi}\varpi^{-m}) & \sum _{y\in c_m}
\psi_{2^{-1}c_{\chi,\psi}\varpi^{-m}}\bigl((y-1)^2-2(y-1)\bigr)\psi_{c_{\chi,\psi} b_{\chi,\psi}\varpi^{-\frac{m+1}{2}}}(y-1)\psi_{c_{\chi,\psi}\varpi^{-m}}(y-1)= \\
\psi(c_{\chi,\psi}\varpi^{-m}) &  \sum _{y\in c_m}
\psi_{2^{-1}c_{\chi,\psi}\varpi^{-m}}\bigl((y-1)^2\bigr)\psi_{c_{\chi,\psi} b_{\chi,\psi}\varpi^{-\frac{m+1}{2}}}(y-1).
\end{split}
\end{equation}
By changing the summation index $y\mapsto 1+a\varpi^{\frac{m-1}{2}}$ we move from summation over $c_m$ to summation over $\K$:
\begin{equation}\nonumber
\begin{split}
& \sum_{y\in c_m}\chi(y) \psi(c_{\chi,\psi}y\varpi^{-m})= \\  &\psi(c_{\chi,\psi}\varpi^{-m})  \sum_{a\in\K }
\psi_{2^{-1}c_{\chi,\psi}\varpi^{-1}}\bigl(a^2)\psi_{c_{\chi,\psi} \varpi^{-1}}(b_{\chi,\psi}a)=\\
 & \psi(c_{\chi,\psi}\varpi^{-m})  \sum_{a\in\K }
\psi_{2^{-1}c_{\chi,\psi}\varpi^{-1}}\bigl(a^2+2b_{\chi,\psi}a)=\\
& \psi(c_{\chi,\psi}\varpi^{-m})\psi_{2^{-1}c_{\chi,\psi}\varpi^{-1}}\bigl(-b^2_{\chi,\psi}\bigr)  \sum_{a\in\K }
\psi_{2^{-1}c_{\chi,\psi}\varpi^{-1}}\bigl((a+b_{\chi,\psi})^2\bigr)=\\
& \psi(c_{\chi,\psi}\varpi^{-m})\psi_{2^{-1}c_{\chi,\psi}\varpi^{-1}}\bigl(-b^2_{\chi,\psi}\bigr)  \sum_{k\in\K }
\psi_{2^{-1}c_{\chi,\psi}\varpi^{-1}}(k^2).
\end{split}
\end{equation}
Since we assume that the residual characteristic of $F$ is odd  it follows from  \eqref{gausspsi} and \eqref{quadgauss} that
$$ \sum_{k\in\K }
\psi_{2^{-1}c_{\chi,\psi}\varpi^{-1}}(k^2)= \eta_{_\K}(2^{-1}c_{\chi,\psi})G_\psi(\K)$$

\end{proof}
\begin{lem} \label{taumult} Suppose that $e(\chi)=m\geq 2$. If $m$ is odd assume in addition that the residual characteristic of $F$ is odd. Fix  $d \in \N$. Assume that  $d$ is relatively prime to the residual characteristic of $F$ and that  $\mu_d \subseteq F$.
\begin{enumerate}
\item If $d$ is odd then $$\tau(\chi ,\psi)^d=\tau(\chi_d ,\psi_d)$$
\item If $d$ is even then $$\tau(\chi ,\psi)^d=\tau(\chi_d ,\psi_d) \times \begin{cases} 1   & m \operatorname{\, is \,even}; \\q^{\frac{-1}{2}} \eta_{_\K}( 2d \cdot c_{\chi,\psi})G_\psi(\K)&   m \operatorname{ \, is \, odd}. \end{cases} $$
\end{enumerate}
\end{lem}
\begin{proof} By Lemma \ref{tauvalue}

\begin{equation}  \nonumber \tau(\chi ,\psi)^d=\chi^d (c_{\chi,\psi})\psi_d(c_{\chi,\psi}\varpi^{-m}) \times \begin{cases} 1   & m \operatorname{ \, is \, even}; \\ \psi_{d2^{-1}c_{\chi,\psi}\varpi^{-1}}\bigl(-b^2_{\chi,\psi}\bigr)\eta_{_\K}^d(2c_{\chi,\psi})q^{\frac {-d}{2}}G_\psi(\K)^d  &   m \operatorname{\, is \, odd}. \end{cases} \end{equation}
Since $d$ is relatively prime to the residual characteristic of $F$ and since $m\geq 2$, $e(\chi)=e(\chi^d)$. Also, by Lemmas \ref{cprop} and \ref{bprop}, $c_{\chi,\psi}=c_{\chi^d,\psi_d}$
and $b_{\chi,\psi}=b_{\chi^d,\psi_d}$. Hence
\begin{equation}  \nonumber \tau(\chi^d ,\psi_d)=\chi^d (c_{\chi,\psi})\psi_d(c_{\chi,\psi}\varpi^{-m}) \times \begin{cases} 1   & m \operatorname{ \, is \, even}; \\ \psi_{d2^{-1}c_{\chi,\psi}\varpi^{-1}}\bigl(-b^2_{\chi,\psi}\bigr)\eta_{_\K}(d 2c_{\chi,\psi})q^{-\half}G_\psi(\K)  &   m \operatorname{\, is \, odd}. \end{cases} \end{equation}
(for the cases where $m$ is odd we have used \eqref{cquadgauss}). Comparing the last two equations we deduce
$$\tau(\chi ,\psi)^d=\tau(\chi_d ,\psi_d) \times \begin{cases} 1   & m \operatorname{ \, is \, even}; \\ \eta_{_\K}(d)\eta_{_\K}^{d-1}( 2c_{\chi,\psi})q^{-\frac{d-1}{2}}G_\psi(\K)^{d-1}&   m \operatorname{\, is \, odd}. \end{cases} $$
The proof for the case where $m$ is even is now completed. We move to the cases where $m$ is odd. In this case we have shown that
$$\frac {\tau(\chi ,\psi)^d} {\tau(\chi^d ,\psi_d)} = \eta_{_\K}(d) \eta_{_\K}^{d-1}( 2c_{\chi,\psi})q^{-\frac{d-1}{2}}G_\psi(\K)^{d-1}.$$
Taking \eqref{quadinv} into account we deduce that
$$\frac {\tau(\chi ,\psi)^d} {\tau(\chi^d ,\psi_d)} = \begin{cases} \eta_{_\K}(d) \eta_{_\K}(-1)^{\frac{d-1}{2}}   & d \operatorname{ \, is \, odd}; \\ \eta_{_\K}(d 2c_{\chi,\psi})q^{-\half}G_\psi(\K)\eta_{_\K}(-1)^{\frac{d-2}{2}}  &   d \operatorname{\, is \, even}. \end{cases}$$
If $d$ is odd the lemma follows from the first item in Lemma \ref{d and sign}. If $d$ is even it is left to show that $\eta_{_\K}(-1)^{\frac{d-2}{2}}=1$. Indeed,  if $d\equiv 2 \, (\operatorname{mod }4)$ then $\frac{d-2}{2}$ is even and if $d\equiv 0 \, (\operatorname{mod }4)$ then $-1 \in {{\K}^*}^2$.
\end{proof}

\section{Main results} \label{mainsec}

\begin{thm} \label{tau theorem} Assume that the residual characteristic of $F$ is odd. Assume also that $d$ is relatively prime to the residual characteristic of  $F$ and that $\mu_d \subseteq F$. Let $J {\subseteq}\widehat{ F^*/{F^*}^d}$ be a lift of $\widehat{ {\K}^*/{\K^*}^d}$. Assume that $\chi^d$ is ramified. Denote $\psi_o=\psi_{\varpi^{e(\psi)}}.$
\begin{enumerate}
\item If $d$ is odd then $$ \prod_{ \eta\in J}{\tau}(\eta\chi ,\psi)=\tau(\chi^d ,\psi_d). $$
\item If $d$ is even then
$$\prod_{\eta \in J}\tau(\chi \eta,\psi)=  \tau(\chi^d,\psi_d) \begin{cases} \eta_{_\K}(c_{\chi,\psi_o}) & e(\chi) \, \operatorname{is} \,  \operatorname{even} ;\\  \eta_{_\K}(2d)q^{-\half}G_\psi(\K)  &  e(\chi) \, \operatorname{is} \,  \operatorname{odd}. \end{cases}$$
\end{enumerate}
\end{thm}
\begin{proof}  If $e(\chi)=1$ then this Theorem is essentially equivalent to Hasse-Davenport product relation, see Corollary \ref{hdtau1}. We move to the case where $e(\chi)>1$.  By Corollary \ref{taulowtwist} we have
$$\prod_{\eta\in J}{\tau}(\eta\chi ,\psi)= {\tau}(\chi,\psi)^d   \prod_{ \eta \in J}\eta(c_{\chi,\psi})=  {\tau}(\eta,\psi)^d
\begin{cases}  1 & d \, \operatorname{is} \,  \operatorname{odd} ;\\  \eta_{_\K}(c_{\chi,\psi_o} ) &  d \, \operatorname{is} \,  \operatorname{even}. \end{cases}$$
The theorem now follows from Lemma \ref{taumult}.
\end{proof}
\begin{remark} From \eqref{c shift prop} it follows that if $a\in \Of^*$ then $\eta_{_\K}(c_{\chi,\psi_o})=\eta_{_\K}(a)\eta_{_\K}(c_{\chi,{\psi_a}_o})$. This implies that for a fixed $\chi$ the sign in the cases where $d$ and $e(\chi)$ are even depends on $\psi$, namely, it is not identically 1.
\end{remark}

\begin{thm} \label{mainres} Assume that the residual characteristic of $F$ is odd. Assume also that $d$ is relatively prime to the residual characteristic of  $F$ and that $\mu_d \subseteq F$. Let $J {\subseteq}\widehat{ F^*/{F^*}^d}$ be a lift of $\widehat{ {\K}^*/{\K^*}^d}$.
\begin{enumerate}
\item If $\chi^d$ is ramified then

\begin{equation} \label{mainepodd}\prod_{\eta \in J} \epsilon(1-s,(\chi\eta)^{-1},\psi)=q^{\frac{d-1}{2} \bigl(e(\psi)-e(\chi^d)\bigr)}\epsilon(1-ds,\chi^{-d},\psi_d).\end{equation}

\item If $\chi^{2d}$ is ramified then
\begin{equation} \label{mainmetaepodd}\prod_{\eta \in J} \widetilde{\epsilon}(1-s,(\chi\eta)^{-1},\psi)=q^{\frac{d-1}{2} \bigl(e(\psi)-e(\chi^{2d})\bigr)}\widetilde{\epsilon}(1-ds,\chi^{-d},\psi_d).\end{equation}

\item If $\chi^2$ is ramified then
\begin{equation} \label{gtilsatog} \prod_{\beta \in \widehat{ F^*/{F^*}^2}}\widetilde{\epsilon}(1-s,\chi^{-1}\beta,\psi)=q^{e(\psi)-e(\chi)}\epsilon^{2} (1-2s,\chi^{-2},{\psi_{_2}}) sign(F)^{e(\chi)}.\end{equation}

\end{enumerate}
\end{thm}

\begin{proof}
We first prove \eqref{mainepodd}. From Proposition \ref{eptau} it follows that
$$\epsilon(1-ds,\chi^{-d},\psi_d)=\bigl(\chi^d(\varpi)q^{\half-ds}\bigr)^{e(\psi_d)-e(\chi^d)}\tau(\chi^d,\psi_d)$$
and that
$$\prod_{\eta \in J} \epsilon(1-s,(\chi\eta)^{-1},\psi)=\prod_{\eta \in J}\bigl(\chi\eta(\varpi)q^{\half-ds}\bigr)^{e(\psi)-e(\chi)}\tau(\chi\eta,\psi).$$
Since $J$ has no elements of order 2, $\prod_{\eta \in G}\eta(\varpi)=1.$ Thus, since  $e(\chi)=e(\chi^d)$, \eqref{mainepodd} follows from the first item in Theorem \ref{tau theorem}.

We now prove \eqref{mainmetaepodd}. Since $d$ is odd, the map $\eta \mapsto \eta^2$ is an automorphism of $J$ . Thus, recalling \eqref{eptildef}, it is sufficient to prove that $\gamma(\psi)^d=\gamma(\psi_d)$ and that
$$\prod_{\eta \in \widehat{{F^*}/{F^*}^d}}\frac{ \epsilon(s+\half,\chi\eta,\psi)}{ \epsilon(2s,\chi^{2}\eta,{\psi_{_2}})}=q^{\frac{d-1}{2} \bigl(e(\psi)-e(\chi^{2d})\bigr)}\frac{ \epsilon(ds+\half,\chi^d,\psi_d)}{ \epsilon(2s,\chi^{2d}\eta,{\psi_{2d}})}.$$
The first assertion is the content of Lemma \ref{weild}. The second assertion follows from  \eqref{mainepodd} and from \eqref{epsilon twist}.

Last{,} we prove \eqref{gtilsatog}{.} Since the residual characteristic of $F$ is odd then $[F^*:{F^*}^2]=4$ and $\gamma(\psi)^4=1$. Thus, due to \eqref{Tate gamma}, this proposition is equivalent to the assertion

$$\prod_{\beta \in \widehat{ F^*/{F^*}^2}}{\epsilon}(s+\half,\chi\beta,\psi)=q^{e(\psi)-e(\chi)}\epsilon^{2} (2s,\chi^{2},{\psi_{_2}}) sign(F)^{e(\chi)}$$
Let $\eta$ be a ramified character of order 2 and let $\eta'$ be an unramified character of order 2. Note that the restriction of $\eta$ to $\Of^*$ is the pull back of $\eta_{_\K}$ and that
 $\eta$ and $\eta'$ generate $\widehat{ F^*/{F^*}^2}$. Hence,
$$\prod_{\beta \in F^*/{F^*}^2}{\epsilon}(s+\half,\chi\beta,\psi)= {\epsilon}(s+\half,\chi,\psi){\epsilon}(s+\half,\chi\eta',\psi){\epsilon}(s+\half,\chi\eta,\psi){\epsilon}(s+\half,\chi\eta \eta',\psi).$$
By \eqref{epsilon twist},
$$\prod_{\beta \in F^*/{F^*}^2}{\epsilon}(s+\half,\chi \beta,\psi)= {\epsilon}(s+\half,\chi,\psi)^2{\epsilon}(s+\half,\chi\eta,\psi)^2.$$
We use Proposition \ref{eptau} and obtain
$$\prod_{\beta \in F^*/{F^*}^2}{\epsilon}(s+\half,\chi \beta,\psi)=\bigl(\chi(\varpi)q^s\bigr)^{4\bigl(e(\psi)-e(\chi)\bigr)}\bigl(\tau(\chi^{-1},\psi)\tau(\chi^{-1}\eta_{_\K},\psi)\bigr)^2$$
and
$${\epsilon}(2s,\chi^{2},\psi)=\bigl(\chi^2(\varpi)q^{2s-\half}\bigr)^{2\bigl(e(\psi)-e(\chi)\bigr)}\tau(\chi^{-2},\psi_2)^2.$$
It remains to show that
$$\bigl(\tau(\chi^{-1},\psi)\tau(\chi^{-1}\eta_{_\K},\psi)\bigr)^2=\tau(\chi^{-2},\psi_2)^2sign(F)^{e(\chi)}.$$
This follows by squaring the $d=2$ case of Theorem \ref{tau theorem} combined with \eqref{quadinv}.
\end{proof}

By Corollary \ref{hdtau1}, if $d$ is odd then the statement given in Theorem \ref{tau theorem} still holds if $\chi^d$ is unramified. However,  \eqref{mainepodd} is false if $\chi^d$ is unramified. In fact if $\chi^d$ is {un}ramified then the product in the left hand side of \eqref{mainepodd} depends on $J$, i.e., on the particular lift of $\widehat{ {\K}^*/{\K^*}^d}$ to $\widehat{ {F}^*/{F^*}^d}$, and in any case the map $s\mapsto \frac{ \prod_{\eta \in J} \epsilon(1-s,(\chi\eta)^{-1},\psi)}{\epsilon(1-ds,\chi^{-d},\psi_d)}^{-1}$ is not constant. Replacing the $\epsilon$-factors in  \eqref{mainepodd} by $\gamma$-factors does not solve the problem. Similar failures occur in the unramified cases of \eqref{mainmetaepodd} and \eqref{gtilsatog} as well.

\section{A representation theoretic application} \label{apsec}
In Sections \ref{coverrep}, \ref{pmea} and \ref{lcm} we recall some of our past work on certain $n$-fold  coverings of ${SL}_2({F})$ and $GL_2(F)$.
We refer the reader to \cite{Sz19}, \cite{GSS} and \cite{GSS2} for details and proofs. While we present some of  our results in greater generality, we ultimately restrict our attention in Section \ref{finalres} to the cases where  $n$ is relatively prime to the residual characteristic of $F$ and is not divisible by 4.  In that Section{, for $p$-adic fields of odd residual characteristic,} we shall state and prove  representation theoretic identities whose ramified cases are equivalent to  \eqref{mainepodd}, \eqref{mainmetaepodd} and \eqref{gtilsatog}, see the first two items in Proposition \ref {unramcomp} and Theorem \ref{ramdet}.

\subsection{Groups and representations} \label{coverrep}

Let $G=GL_2(F)$ and let  $G_o={{SL}_2({F})}$ be its derived group. For a subset $A$ of $G$ we define $A_o=A\cap G_o$. Let $N \cong {\ F}$ be the group of upper triangular unipotent matrices in $G$ and let $H$ be the subgroup of diagonal elements inside $G$. Denote $B=H \ltimes N$. We have  $B_o=H_o \ltimes N$.

Fix  $n \in \N$ and $c \in \Z$. Denote
$$n_c=\frac{n}{gcd(n,4c+1)}, \quad d=\begin{cases} n &  n \, \operatorname{is} \,  \operatorname{odd }; \\  \frac {n} {2}  &  n \, \operatorname{is} \,  \operatorname{even. }  \end{cases}, \quad d_c=\frac{d}{gcd(d,4c+1)}=\begin{cases} n_c&  n \, \operatorname{is} \,  \operatorname{odd }; \\  \frac {n_c} {2}  &  n \, \operatorname{is} \,  \operatorname{even. }  \end{cases}$$

We assume that $\mu_n \subseteq F^*$. Let $\widetilde{G}^{(n),c}=\widetilde{G}$ be the $c$-twisted $n$-fold cover of $G$ constructed by Kazhdan and Patterson in \cite{KP}. This is the same topological central extension of $G$ by $\mu_n$ introduced in Section 8.1 of \cite{GSS2}.  We have the short exact sequence
$$1\rightarrow \mu_n \rightarrow \widetilde{G}  \rightarrow  G \rightarrow  1.$$
$G$ acts on $\widetilde{G}$ by conjugations: for $h\in G$  and $\widetilde{g} \in \widetilde{G}$ we define $\widetilde{g}^h=\widetilde{h}\widetilde{g}\widetilde{h}^{-1}$, where $\widetilde{h}$ is any inverse image of $h$ in $\widetilde{G}$.

We shall denote by $\widetilde{A}$ the pre-image in $\widetilde{G}$ of a subset $A$ of $G$.  We note that  $\widetilde{G_o}$ is the derived group of $\widetilde{G}$ and that is it independent of $c$.

We fix an embedding $\mu_n \hookrightarrow \C^*$. Let $C$ be subgroup of $G$. A representation $\pi$ of $\widetilde{C}$ is called genuine if  $\mu_n$ acts via this embedding. If $g \in G$ normalizes  $C$ then $c\mapsto \pi(g^{-1}cg)$ is another genuine representation of $\widetilde{C}$. We denote it by $\pi^g$.

It is a fact that $\widetilde{G}$ splits uniquely over $N$. Hence we may identify this group with its embedding in  $\widetilde{G_o}$. Moreover, we have
${\widetilde{B_o}}={\widetilde{H_o}} \ltimes N,  \, {\widetilde{B}}={\widetilde{H}} \ltimes N$. Hence, as in the linear case we may extend any representation of  $\widetilde{H}$ ({representation of} $\widetilde{H_o})$ to { a representation of} $\widetilde{B}$ ({representation of} $\widetilde{B_o}$) by defining it to be trivial on $N$. We shall not distinguish between representations of  $\widetilde{B}$ ({representations of} $\widetilde{B_o}$) and representations of  $\widetilde{H}$ ({representations of }$\widetilde{H_o}$).

From this point  $(\sigma,V)$ and $(\sigma_o,V_o)$ will denote  genuine smooth irreducible representations of $\widetilde{H}$ and $\widetilde{H_o}$ respectively. Both representations are finite dimensional. Precisely,  $$\dim V_o=\sqrt{[F^*:{F^*}^d]}, \, \, \dim V=\sqrt{[F^*:{F^*}^n][F^*:{F^*}^{n_c}]}.$$ The isomorphism classes of $\sigma$ and $\sigma_o$ are determined by their central characters. The genuine characters of $Z(\widetilde{H_o})$, the center of $\widetilde{H_o}$, are parameterized by the set of $d^{th}$ powers of the characters of $F^*$ . If $\sigma_o$ is mapped to $\chi^d$ via this parametrization we say that $\chi^d$ corresponds to $\sigma_o$. We note that this parametrization is canonical unless $n\equiv 2 \, (\operatorname{mod }4)$. In this case the parametrization depends on an additive character of $F$, and we fix $\psi$ to be this character. See Section 4.2 of \cite{Sz19} for exact details.

 If $n \equiv 2 \, (\operatorname{mod }4)$ then for $\beta \in \widehat{F^*/{F^*}^2}$ we denote by $\beta\sigma_o$ the  genuine smooth irreducible representation of $\widetilde{H_o}$ whose corresponding character is $(\beta\chi)^d=\beta\chi^d.$

Clearly, $Z(\widetilde{H})\cap \widetilde{G_o}  \subseteq Z(\widetilde{H_o})$. This inclusion is strict if and only if $n$ is even. For all $n$, the set of genuine characters of $Z(\widetilde{H})\cap \widetilde{G_o} $ is parameterized canonically by the set of characters of  $n^{th}$ powers of the characters of $F^*$, see Section 10.1 of \cite{GSS2}. If the restriction of the central character of $\sigma$ ({ the central character of} $\sigma_o)$ to $Z(\widetilde{H})\cap \widetilde{G_o} $ is mapped to $\chi^n$ under this parametrization we say that $\chi^n$ is the linear character associated with  $\sigma$ ({ with} $\sigma_o$). If $n$ is odd the linear character of the associated with $\sigma_o$  is equal to the character corresponding to $\sigma_o$ . If $n$ is even then  the linear character associated with $\sigma_o$  is the square of  the character corresponding to $\sigma_o$.

Let $\delta$ be the modular character of $H$ defined by $diag(a,b) \mapsto \ab \frac a b \ab^{\half}$.  We shall continue to denote by $\delta$ the non-genuine character of  $\widetilde{H}$ defined by setting $\delta(\widetilde{h})=\delta(h)$ where  $h \in H$ and $\widetilde{h}$ is an inverse image in $\widetilde{H}$ of $h$. We shall denote by $\delta_o$ be the restriction of $\delta$ to $\widetilde{H_o}$.

For $s\in \C$ we now define the normalized parabolic inductions
$$I(\sigma,s)=Ind_{\widetilde{B}}^{\widetilde{G}} \delta^{s+\half}\otimes \sigma, \, \, I({\sigma_o},s)=Ind_{\widetilde{B_o}}^{\widetilde{G_o}} \delta_o^{s+\half}\otimes \sigma_o.$$

\begin{lem} \label{restlema} {\cite{GSS2}, Theorem 8.15}
\begin{enumerate}
\item $I(\sigma_o,s)$ appears in the restriction of $I(\sigma,s)$ to $\widetilde{G_o}$ if and only if the  linear characters associated with $\sigma$ and $\sigma_o$ are equal.
\item Suppose that $I(\sigma_o,s)$ appears in the restriction of $I(\sigma,s)$ to $\widetilde{G_o}$
\begin{enumerate}

\item If $n$ is odd then
$$I(\sigma,s)\mid_{\widetilde{G_o}}=n_c\ab n_c \ab^{-\half} I(\sigma_o,s).$$
\item If  $n\equiv 2 \, (\operatorname{mod }4)$ then
$$I(\sigma,s)\mid_{\widetilde{G_o}}=\bigoplus_{\widehat{F^*/{F^*}^2}} d_c\ab d_c \ab^{-\half} I(\beta\sigma_o,s).$$
\end{enumerate}
\end{enumerate}
\end{lem}

We note that if $gcd(p,n)=1$ then $\widetilde{G}$ splits over $GL_2(\Of)$. Hence we have the notion of unramified genuine principal series representations. It is a fact that  $I(\sigma_o,s)$ appears in the restriction of some unramified genuine principal series representation of $\widetilde{G}$ if and only if the linear character associated with $\sigma_o$ is unramified. We shall not use this fact.

\subsection{Plancherel measures} \label{pmea}

Let $w_{_0}=\left( \begin{array}{cc} {0} & {1} \\ {-1} & {0} \end{array} \right)$ be a representative of the non-trivial Weyl element in $G_o$. Consider now the standard intertwining operators involving integration over $N${:}

$$A(\sigma,s):I(\sigma,s) \rightarrow I(\sigma^w,{-s}), \quad A_o({\sigma}_o,s):I(\sigma_o,s) \rightarrow I(\sigma_o^w,{-s}).$$

See Section 9.1 of \cite{GSS2} for exact definitions. The Plancherel measures $\mu(\sigma,s)$ and   $\mu(\sigma_o,s)$ are the rational functions in $q^{-ns}$ defined by the relations
$$A(\sigma^w,{-s})\circ A(\sigma_s)=\mu(\sigma,s)^{-1}Id, \quad A(\sigma_o^w,{-s})\circ A(\sigma_o,s)=\mu(\sigma_o,s)^{-1}Id.$$

Both the intertwining operators and the Plancherel measures depend on $e(\psi)$ via a choice of Haar measures. Following the conventions in the literature we suppress this dependence.

Since $N \subseteq \widetilde{G_o}$, the intertwining operators commute with the restriction from $ \widetilde{G}$ to $ \widetilde{G_o}$. The following lemma is a direct consequence of this fact.
\begin{lem} \label{oneplan} (\cite{GSS2}, Corollary 10.2). If the linear characters associated with $\sigma_o$ and $\sigma$ are equal then $\mu(\sigma_o,s)=\mu(\sigma,s)$.
\end{lem}
The Plancherel measures associated with $I(\sigma_o,s)$  was computed in \cite{GoSz} and \cite{Sz19}. Here we just recall the formulas in the cases where $gcd(p,n)=1$.
\begin{lem} \label{GoSz lem} (\cite{GoSz}, Theorem 5.1)
Assume that $gcd(p,n)=1$. Assume that the linear character associated with $\sigma_o$ is $\chi^n$.
$$\mu(\sigma_o,s)^{-1}= q^{e(\psi)-e(\chi^n)}\frac{L \bigl(ns,\chi^n \bigr)L \bigl(-ns,\chi^{-n}\bigr)}{L \bigl(1-ns,\chi^{-n} \bigr)L \bigl(1+ns,\chi^{n}\bigr)}.$$
\end{lem}

Using the last two lemmas we identify the powers of $q$ appearing in Theorem  \ref{mainres} as powers of Plancherel measures:
\begin{cor} \label{tobeused}
Assume that $gcd(p,n)=1$. Assume that the linear character associated with  both $\sigma$ and $\sigma_o$ is $\chi^n$. If $\chi^n$ is ramified then $$\mu(\sigma_o,s)^{-1}=\mu(\sigma,s)^{-1}=q^{e(\psi)-e(\chi^n)}.$$
\end{cor}

\subsection{Local coefficient matrices and a related invariant} \label{lcm}
The map $x\mapsto \left( \begin{array}{cc} {1} & {x} \\ {0} & {1} \end{array} \right)$ is an isomorphism from $F$ to $N$. In particular, $n(x) \mapsto \psi(x)$ is a character of $N$. We shall continue to denote it by $\psi$. Let $\pi$ be a representation of  $\widetilde{G}$ or $\widetilde{G_o}$. The space $Hom_N(\pi,\psi)$ is called the space of $\psi$-Whittaker functionals on $\pi$ and is denoted by $Wh_\psi(\pi)$.  Unlike the linear case, a Whittaker functional is generally not unique, \cite{GSS}, Section 2. In particular,
$$\dim Wh_\psi\bigl(I(\sigma,s)\bigr)=\dim V, \, \dim Wh_\psi\bigl(I(\sigma_o,s)\bigr)=\dim V_o.$$
By duality, the standard intertwining operators $A(\sigma,s)$ and $A(\sigma_o,s)$ give rise to linear maps between  spaces of $\psi$-Whittaker functionals:
$$A_\psi(\sigma,s):Wh_\psi\bigl(I(\sigma^w,-s)\bigr) \rightarrow Wh_\psi\bigl(I(\sigma,s)\bigr),$$ $${A_o}_{\psi}(\sigma_o,s):Wh_\psi\bigl(I(\sigma_o^w,-s)\bigr) \rightarrow Wh_\psi\bigl(I(\sigma_o,s)\bigr).$$

The spaces, $Wh_\psi\bigl(I(\sigma_o^w,-s)\bigr)$ and  $Wh_\psi\bigl(I(\sigma_o,s)\bigr)$ are canonically identified with $\widehat{V_o}$, the space of linear  functionals on $V_{o}$ by means of Jacquet-type integrals, see Section 3.2 of \cite{GSS2}. Using these identifications, ${A_o}_\psi(\sigma_o,s)$ is viewed as a linear operator acting on a linear space. A matrix representing this map is called a local coefficients matrix associated with $\sigma_o$ and $\psi$. Its characteristic polynomial is an invariant of the inducing datum $\sigma_o$ and $s$ and the Whittaker character $\psi$. For the same reasons we also have a local coefficients matrix associated with $\sigma$ and $\psi$.

Denote by $D_o(\sigma_o,s,\psi)$ and $D(\sigma,s,\psi)$
the determinant of ${A_o}_\psi(\sigma_o,s)$  and $A_\psi(\sigma,s)$ respectively. In the linear case,  namely in the $n=1$ case, these are equal to the reciprocal of Shahidi's local coefficients, \cite{Shabook}. Hence, it is expected that  $D_o(\sigma_o,s,\psi)$ and $D(\sigma,s,\psi)$ are related to $\gamma$-factors. The following relations between $D(\sigma,s,\psi)$ and $D_o(\sigma_o,s,\psi)$  is a direct consequence of Lemma \ref{restlema} combined with the fact that the intertwining operators commute with the restriction from
$\widetilde{G}$ to $ \widetilde{G_o}$:
\begin{prop} \label{reldet} (\cite{GSS2}, Theorem 10.9) Suppose that $\sigma_o$ appears in the restriction of $\sigma$ to $\widetilde{H_o}$.
\begin{enumerate}
\item If $n$ is odd then
$$D(\sigma,s,\psi)=D_o(\sigma_o,s,\psi)^{n_c \ab n_c \ab^{-\half}}.$$
\item If  $n\equiv 2 \, (\operatorname{mod }4)$ then
$$D(\sigma,s,\psi)= \prod_{\widehat{F^*/{F^*}^2}}D_o(\beta\sigma_o,s,\psi) ^{d_c\ab d_c \ab^{-\half}}.$$
\end{enumerate}
\end{prop}

\subsection{$D(\sigma,s,\psi)$ and $D_o(\sigma_o,s,\psi)$ in the tame cases}  \label{finalres}
In this section we assume that  $gcd(n,p)=1$ and that  $n$ is not divisible by 4. Also, we shall assume $I(\sigma_o,s)$ appears in the restriction of $I(\sigma)$ to $\widetilde{G_o}$. Accordingly, we shall assume that $\chi^d$ corresponds to $\sigma_o$ and that linear character associated with  both $\sigma$ and $\sigma_o$ is $\chi^n$.

\begin{prop} \label{unramcomp} Assume that $\psi$ is normalized and that $\chi^n$ is unramified.
\begin{enumerate}

\item \begin{equation} \label{degoeq}D_o(\sigma_o,s,\psi)=\mu(\sigma_o,s)^{\frac{1-d}{2}} \begin{cases} \gamma(1-ds,\chi^{-d},\psi_{d})  &  n \,  \operatorname{is} \,  \operatorname{odd}; \\ \widetilde{\gamma}(1-ds,\chi^{-d},\psi_{d}) &   n \equiv 2 \, (\operatorname{mod }4). \end{cases} \end{equation}

\item   If $n \equiv 2 \, (\operatorname{mod }4)$ then
\begin{equation}  \label{passdet} \prod_{\widehat{F^*/{F^*}^2}}D_o(\beta\sigma_o,s,\psi)=\mu(\sigma_o,s)^{1-n}\gamma^2(1-ns,\chi^{-n},\psi_{n}). \end{equation}

\item
\begin{equation}  \label{deg} D(\sigma, s, \psi)=\epsilon(n) \mu(\sigma, s)^{\frac{(1-n)n_c}{2}}  \cdot \gamma(1-ns, \chi^{-n}, \psi_{n})^{n_c} \end{equation}
where $$\epsilon(n)= \begin{cases} sign(F) &  n \equiv 2 \, (\operatorname{mod }4); \\ 1  &   \operatorname{otherwise}. \end{cases}$$
\end{enumerate}
\end{prop}
{\begin{proof}
The proof of this Theorem is based on the explicit computation of the local coefficients matrices in \cite{Sz19}, Section 4.4. The first item was proved in \cite{GSS}, Theorem 3.14  for the case where $\chi^d$ is unramified. In \cite{GSS}, Theorem 3.14 we had $\psi$ on the right hand side rather than $\psi_d$. As in Lemma \ref{verifygao}, the change is justified by \eqref{changepsi}. To finish the proof of the first item it is left to consider the case where $n \equiv 2 \, (\operatorname{mod }4)$ and  $\chi^d$ in ramified. From our assumptions it follows that $p$ is odd and that $\chi^d\mid_{\Of^*}=\eta_{_\K}$. This case was handled in \cite{GSS2}, Theorem 9.12. In that theorem also we had  $\psi$ on the right hand side rather than $\psi_d$. The change is justified by \eqref{changepsi} and by the first item in Lemma \ref {d and sign} (note that $\frac{d-1}{2}$ is even). The second item follows from \eqref{degoeq} combined with Lemmas \ref{verifygao} and \ref{GoSz lem} and by using the fact that $\widehat{F^*/{F^*}^2}$ has 4 elements. Proposition \ref{reldet} and \eqref{degoeq} prove the third item for odd $n$. For the $n\equiv 2 \, (\operatorname{mod }4)$ case one uses \eqref{passdet} as well.
\end{proof}}
The assumption that $\psi$ is normalized in Proposition \ref{unramcomp} can be removed by making obvious changes in the proofs in \cite{Sz19}, \cite{GSS}, \cite{GSS2}. However, our goal is to show that {for $p$-adic fields of odd residual characteristic,} this proposition hold  in the ramified cases with no restriction on $\psi$  (up to a sign in the $n \equiv 2 \, (\operatorname{mod }4)$ cases).

To compute $D_o(\sigma_o,s,\psi)$ in the ramified cases we shall need an explicit description of the local coefficients matrices. This description is  given in \cite{Sz19}, Section 4.4  under the assumption that $\psi$ is normalized. These matrices are simpler than those computed in the unramified case. We now drop the assumption on $\psi$. For this purpose we pick $J$ to be the specific lift of $\widehat{\K^*/{\K^*}^d}$ described in \cite{Sz19}, Example 2.4. Since  $d$ is odd this just means that $\eta(\varpi)=1$ for all $\eta \in J$. We denote by $\eta_o$ the generator of $J$ used in  \cite{Sz19}, Section 2.4 (in the notation of that paper $\eta_o=\eta_\varpi$) and we denote $\widetilde{\eta_o}=\eta^{\frac {d+1}{2}}$. Observe that $\widetilde{\eta_o}$ is another generator of $J$.
\begin{lem} \label{goramlcm} Assume that $\chi^n$ is ramified.
\begin{enumerate}
\item Suppose that $n$ is odd. The $d \times d$ matrix whose $(i,j)$ coordinate equals
$$\tau(i,j,\chi,s,\psi)= \begin{cases} \epsilon(1-s,\chi^{-1}\eta_o^{i+j},\psi)  & i-j\equiv e(\chi)-e(\psi) \, (\operatorname{mod }{ d}) ; \\ \\ 0  & otherwise \end{cases}$$
is a local coefficients matrix associated with $\sigma_o$ and $\psi$.
\item Suppose that $n  \equiv 2 \, (\operatorname{mod }4)$.  The $d \times d$ matrix whose $(i,j)$ coordinate equals
$$\widetilde{\tau}(i,j,\chi,s,\psi)= \begin{cases} \widetilde{\epsilon}(1-s,\chi^{-1}\widetilde{\eta_o}^{i+j},\psi)  & i-j\equiv e(\chi)-e(\psi) \, (\operatorname{mod }{ d}) ; \\ \\ 0  & otherwise \end{cases}$$
is a local coefficients matrix associated with $\sigma_o$ and $\psi$.
\end{enumerate}
\end{lem}
\begin{proof}
We use the notation of \cite{Sz19}. Suppose first that $n$ is odd. By Theorem 4.12 of \cite{Sz19},
$$ \tau(i,j,\chi,s,\psi)= \gamma_{_J}(1-s,\chi^{-1}\eta_o^{i+j},\psi,k^{i-j}).$$
Thus, one need to show that
$$\gamma_{_J}(1-s,\chi^{-1},\psi, k^{t})=\begin{cases} \epsilon(1-s,\chi^{-1},\psi)  & t\equiv e(\chi)-e(\psi) \, (\operatorname{mod }{ d}) ; \\ \\ 0  & otherwise .\end{cases}$$
The proof goes almost word for word as the ramified portion of the proof of Proposition 4.16 in \cite{Sz19}. In fact, one only need to replace $e(\chi)$ be $e(\chi)-e(\psi)$ in the last two lines of Page 144 of \cite{Sz19}.
For the $n  \equiv 2 \, (\operatorname{mod }4)$ case one uses similar argument, replacing $\gamma$ with $\widetilde{\gamma}$, $\epsilon$ with $\widetilde{\epsilon}$ and Proposition 4.16 in \cite{Sz19} with Proposition 4.18 in \cite{Sz19}.
\end{proof}
We are now ready to state and prove the desired ramified analog of Proportion  \ref{unramcomp}.

\begin{thm} \label{ramdet} Assume that the residual characteristic of $F$ is odd and that $\chi^n$ is ramified.
\begin{enumerate}

\item \begin{equation}\label{ramdego}  D_o(\sigma_o,s,\psi)=\mu(\sigma_o,s)^{\frac{1-d}{2}} \begin{cases} \gamma(1-ds,\chi^{-d},\psi_{ d})  &  n \,  \operatorname{is} \,  \operatorname{odd}; \\ \widetilde{\gamma}(1-ds,\chi^{-d},\psi_{ d}) &   n \equiv 2 \, (\operatorname{mod }4). \end{cases} \end{equation}
\item   If $n \equiv 2 \, (\operatorname{mod }4)$ then
\begin{equation}  \label{passdetram} \prod_{\widehat{F^*/{F^*}^2}}D_o(\beta\sigma_o,s,\psi)=sign(F)^{e(\chi^n)}\mu(\sigma_o,s)^{1-n}\gamma^2(1-ns,\chi^{-n},\psi_{ n}). \end{equation}
\item $$D(\sigma, s, \psi)=sign(\sigma) \mu(\sigma, s)^{\frac{(1-n)n_c}{2}}  \cdot \gamma(1-ns, \chi^{-n}, \psi_{ n})^{n_c}$$
where $$sign(\sigma)= \begin{cases} sign(F)^{e(\chi^n)} &  n \equiv 2 \, (\operatorname{mod }4); \\ 1  &  n \,  \operatorname{is \, odd}. \end{cases}$$

\end{enumerate}
\end{thm}
\begin{proof} We use Lemma \ref{goramlcm} and the fact that  both $\eta$ and $\widetilde{\eta_o}$ generate $J$.
\begin{equation}\nonumber
\begin{split}  D_o(\sigma_o,s,\psi)= & \begin{cases} \prod_{i=0}^{n-1} \epsilon(1-s,(\chi\eta_o^i)^{-1},\psi)  &  n \,  \operatorname{is} \,  \operatorname{odd}; \\ \prod_{i=0}^{n-1} \widetilde{\epsilon}(1-s,(\chi\widetilde{\eta_o}^i)^{-1},\psi) &   n \equiv 2 \, (\operatorname{mod }4) \end{cases}\\ =& \begin{cases} \prod_{\eta \in J} \epsilon(1-s,(\chi\eta)^{-1},\psi)  &  n \,  \operatorname{is} \,  \operatorname{odd}; \\ \prod_{\eta \in J} \widetilde{\epsilon}(1-s,(\chi\eta)^{-1},\psi) &   n \equiv 2 \, (\operatorname{mod }4) .\end{cases} \end{split} \end{equation}
The first item in this  Theorem now follows from  \eqref{mainepodd} and \eqref{mainmetaepodd} along with Corollary \ref{tobeused}. For the second item use the fact that $\widehat{F^*/{F^*}^2}$ has 4 elements and combine \eqref{ramdego} with \eqref{gtilsatog}. The odd case in the third item of this theorem follows from the first item and from Proposition \ref{reldet}. For the $n \equiv 2 \, (\operatorname{mod }4)$ case we also use the second item.
\end{proof}

\subsection{A remark on the even cases}
While Theorem \ref{tau theorem} is sufficient for proving even analogs to \eqref{mainepodd} and \eqref{mainmetaepodd} we chose not to explore them in this paper as these analogs are connected to the representation theory of the $n$-fold cover of $SL_2(F)$ where $n$ is divisible by 4. The structure of this group is { different} than the structure of the other covering groups studied here. We refer the reader to \cite{GaoGurKar1}  Section 6.1, for details. We also note that the local coefficients matrices for this group are { different} than those of the local coefficients matrices for the other coverings of $SL_2(F)$ discussed in this paper, see \cite{GSS2}, Section 9.4. In the unramified cases, the determinant of the local coefficients  matrix for that group was computed in  \cite{GSS2}, Theorem 9.13 and neither $\gamma$ nor $\widetilde{\gamma}$ appears there. Moreover, a computation of one ramified example in that theorem suggests that the general form of this determinant is somehow mysterious. We hope that a better understating of the representation theory of that group would also shed a light on the correct even analogs to  \eqref{mainepodd} and \eqref{mainmetaepodd}. Last we note that we believe that when $n$ is divisible by 4, \eqref{deg} should be true in the ramified cases with no sign involved. We hope to discuss this in a future work.

{{\bf Acknowledgments.}\\ I would like to thank Fan Gao, Gil Alon, Ofer Hadas, Nadya Gurevich and Amalia Szpruch-Musih for valuable discussions on the subject matter. I would also like to thank the referee for the  enlightening comments.}

{{\bf Funding.}\\ This research was supported by the Israel Science Foundation (grant No. 1643/23).}

\bibliography{hd}
\bibliographystyle{acm}

\end{document}